\documentclass[12pt]{article}

\usepackage{amsthm,amsfonts, amsbsy, amssymb,amsmath,graphicx}
\usepackage{graphics}
\usepackage[english]{babel}
\usepackage{graphicx}
\usepackage{lineno}
\usepackage{enumerate}
\usepackage{tikz}
\usepackage{hyperref}
\usepackage{color}
\usepackage{tikz}

\hypersetup{colorlinks=true, linkcolor=blue, citecolor=red,urlcolor=blue}

\newtheorem{remark}{Remark}[section]

\newtheorem{theorem}[remark]{Theorem}
\newtheorem{proposition}[remark]{Proposition}
\newtheorem{observation}[remark]{Observation}

\newtheorem{corollary}[remark]{Corollary}

\title{Total Roman $\{2\}$-domination in graphs \footnote{This is an original manuscript of an article published by Taylor \& Francis in Quaestiones Mathematicae, available online: http://www.tandfonline.com/10.2989/16073606.2019.1695230.}}

\author{Suitberto Cabrera Garc\'ia$^{(1)}$, Abel Cabrera Mart\'inez$^{(2)}$,\\ Frank A. Hern\'andez Mira $^{(3)}$, Ismael G. Yero $^{(4)}$\\
\\
$^{(1)}$ {\small Universitat Polit\'ecnica de Valencia}\\{\small Departamento de Estad\'istica e Investigaci\'on Operativa Aplicadas y Calidad}\\{\small Camino de Vera s/n, 46022 Valencia, Spain.}\\ {\small suicabga\@@eio.upv.es}  \\
$^{(2)}${\small Universitat Rovira i Virgili}\\
{\small Departament d'Enginyeria Inform\`atica i Matem\`atiques } \\  {\small Av. Pa\"{\i}sos
Catalans 26, 43007 Tarragona, Spain.} \\{\small
  abel.cabrera\@@urv.cat}\\
$^{(3)}$ {\small Universidad Aut\'onoma de Guerrero}\\ {\small Facultad de  Matem\'{a}ticas, Carlos E. Adame 5, La Garita, 39350 Acapulco, Mexico}\\
{\small fmira8906\@@gmail.com}\\
$^{(4)}$ {\small Universidad de C\'adiz,}
{\small Departamento de Matem\'aticas}\\{\small Av. Ram\'on Puyol s/n, 11202 Algeciras, Spain.}\\ {\small
ismael.gonzalez\@@uca.es
}}

\date{}

\begin{document}

\maketitle

\begin{abstract}
Given a graph $G=(V,E)$, a function $f:V\rightarrow \{0,1,2\}$ is a total Roman $\{2\}$-dominating function if

\vspace{.08cm}

\begin{itemize}
\item every vertex $v\in V$ for which $f(v)=0$ satisfies that $\sum_{u\in N(v)}f(u)\geq 2$, where $N(v)$ represents the open neighborhood of $v$, and
\item every vertex $x\in V$ for which $f(x)\geq 1$ is adjacent to at least one vertex $y\in V$ such that $f(y)\geq 1$.
\end{itemize}

\vspace{.3cm}

The weight of the function $f$ is defined as $\omega(f)=\sum_{v\in V}f(v)$. The total Roman $\{2\}$-domination number, denoted by $\gamma_{t\{R2\}}(G)$, is the minimum weight among all total Roman $\{2\}$-dominating functions on $G$. In this article we introduce the concepts above and begin the study of its combinatorial and computational properties. For instance, we give several closed relationships between this parameter and other domination related parameters in graphs. In addition, we prove that the complexity of computing the value $\gamma_{t\{R2\}}(G)$ is NP-hard, even when restricted to bipartite or chordal graphs.
\end{abstract}


\vspace{.1cm}


\section{Introduction}

Throughout this article we only consider simple graphs $G$ with vertex set $V(G)$ and edge set $E(G)$. That is, graphs that
are finite, undirected, and without loops or multiple edges. Given a vertex $v$ of $G$, $N_G(v)$ denotes the \emph{open neighborhood} of $v$ in $G$: $N_G(v)=\{u\in V(G): uv\in E(G)\}$. The \emph{closed neighborhood}, denoted by $N_G[v]$, equals $N_G(v) \cup \{v\}$. Whenever possible, we shall skip the subindex $G$ in the notations above.

A function $f: V(G) \rightarrow \{0,1,2, \ldots\}$ on $G$ is said to be a \emph{dominating function} if for every vertex $v$ such that $f(v)=0$, there exists a vertex $u\in N(v)$, such that $f(u)>0$; furthermore, $f$ is said to be a \emph{total dominating function} (TDF) if for every vertex $v$, there exists a vertex $u\in N(v)$, such that $f(u)>0$. The \emph{weight} of a function $f$ on a set $S\subseteq V(G)$ is $f(S)=\sum_{v\in S}f(v)$. If particularly $S=V(G)$, then $f(V(G))$ will be represented as $\omega(f)$.

Recently, (total) dominating functions in domination theory have received much attention. A purely theoretic motivation is given by the fact that the (total) dominating function problem can be seen, in some sense, as a proper generalization of the classical (total) domination problem. That is, a set $S\subseteq V(G)$ is a (\emph{total}) \emph{dominating set} if there exists a (total) dominating function $f$ such that $f(x)>0$ if and only if $x\in S$. The (\emph{total}) \emph{domination number} of
$G$, denoted by ($\gamma_t(G)$) $\gamma(G)$, is the minimum cardinality among all (total) dominating sets of $G$, or equivalently, the minimum weight among all (total) dominating functions on $G$. Domination in graphs is a classical topic, and nowadays, one of the most active areas of research in graph theory. For more information on domination and total domination see the books \cite{book1, book2, book-total-dom} and the survey \cite{Henning2009}.

From now on, we restrict ourselves to the case of functions $f: V(G) \rightarrow \{0,1,2\}$. Let $V_i=\{v\in V(G): f(v)=i\}$ for every $i\in \{0,1,2\}$. We will identify $f$ with the three subsets of $V(G)$ induced by $f$ and write $f(V_0, V_1, V_2)$. Notice that the weight of $f$ satisfies $\omega(f)=\sum_{i=0}^2 i|V_i|=2|V_2|+|V_1|$.  We shall also write $V_{0,2}=\{ v\in V_0 \colon N(v)\cap V_2 \neq \emptyset\}$ and $V_{0,1}= V_0\setminus V_{0,2}$.

We now define some types of (total) dominating functions, which are obtained by imposing certain restrictions, and introduce a new
one, in order to begin with the exposition of our results.

A \emph{Roman $\{2\}$-dominating function} (R2DF) is a dominating function $f(V_0,V_1,V_2)$ satisfying the condition that for every vertex $v\in V_0$, $f(N(v))\geq 2$. The \emph{Roman $\{2\}$-domination number} of $G$, denoted by $\gamma_{\{R2\}}(G)$, is the minimum weight among all R2DFs on $G$. A R2DF of weight $\gamma_{\{R2\}}(G)$ is called a $\gamma_{\{R2\}}(G)$-function. This concept was introduced by Chellali et al. in \cite{chellali2016}. It was also further studied in \cite{henning2016}, where it was called \emph{Italian domination number}.

A \emph{total Roman dominating function} (TRDF) on a graph $G$ is a TDF $f(V_0,V_1,V_2)$ on $G$ satisfying that for every vertex $v\in V_0$ there exists a vertex $u\in N(v)\cap V_2$. The \emph{total Roman domination number}, denoted by $\gamma_{tR}(G)$, is the minimum weight among all TRDFs on $G$. A TRDF of weight $\gamma_{tR}(G)$ is called a $\gamma_{tR}(G)$-function. This concept was introduced by Liu and Chang \cite{Liu2013}. For recent results on the total Roman domination in graphs we cite \cite{AbdollahzadehAhangarHenningSamodivkinEtAl2016, TRDF2019, Dorota2018}.

A set $S\subseteq V(G)$ is a \emph{double dominating set} of $G$ if for every vertex $v\in V(G)$, $|N[v]\cap S|\geq 2$. The \emph{double domination number} of $G$, denoted by $\gamma_{\times 2}(G)$, is the minimum cardinality among all double dominating sets of $G$. This graph parameter was introduced in \cite{haynes2000} by Harary and Haynes, and it was also studied, for example, in \cite{chellali2006, chellali2005, harant2005}.

In this article we introduce the study of total Roman $\{2\}$-domination in graphs. We define a \emph{total Roman $\{2\}$-dominating function} (TR2DF) to be a R2DF on $G$ which is a TDF as well. The \emph{total Roman $\{2\}$-domination number},
denoted by $\gamma_{t\{R2\}}(G)$, is the minimum weight among all TR2DFs on $G$.

In particular, we can define a double dominating function (DDF) to be a TR2DF $f(V_0, V_1, V_2)$ in which $V_2=\emptyset$. Obviously $f(V_0, V_1, \emptyset)$ is a DDF if and only if $V_1$ is a double dominating set of $G$.

To illustrate the definitions above, we consider the graph shown in Figure \ref{fig1}.

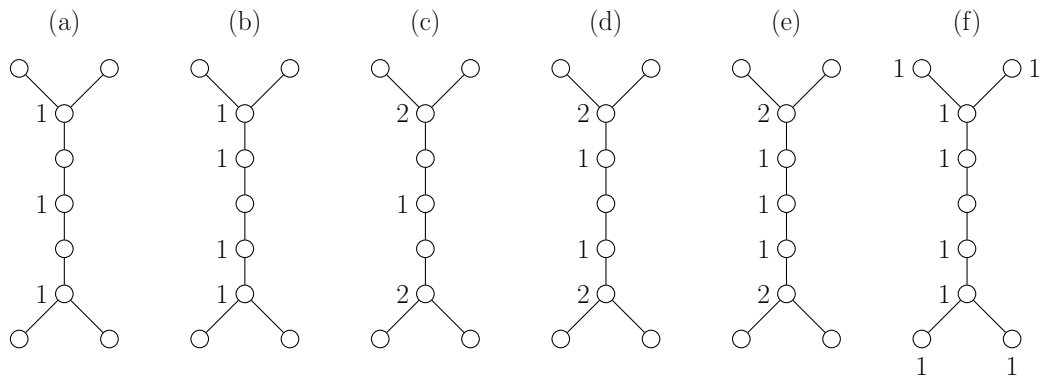
\begin{figure}[ht]
\centering
\begin{tikzpicture}[scale=.6, transform shape]


\node [draw, shape=circle] (b2) at  (0,1) {};
\node at (-0.5,1) {\Large $1$};

\node [draw, shape=circle] (b3) at  (0,2) {};

\node [draw, shape=circle] (b4) at  (0,3) {};
\node at (-0.5,3) {\Large $1$};

\node [draw, shape=circle] (b5) at  (0,4) {};

\node [draw, shape=circle] (b6) at  (0,5) {};
\node at (-0.5,5) {\Large $1$};

\node [draw, shape=circle] (b61) at  (-1,6) {};
\node [draw, shape=circle] (b62) at  (1,6) {};

\node at (0,7) {\Large (a)};

\node [draw, shape=circle] (b21) at  (-1,0) {};
\node [draw, shape=circle] (b22) at  (1,0) {};

\draw(b2)--(b3)--(b4)--(b5)--(b6);
\draw(b21)--(b2)--(b22);
\draw(b61)--(b6)--(b62);


\node [draw, shape=circle] (b2) at  (4,1) {};
\node at (3.5,1) {\Large $1$};

\node [draw, shape=circle] (b3) at  (4,2) {};
\node at (3.5,2) {\Large $1$};

\node [draw, shape=circle] (b4) at  (4,3) {};

\node [draw, shape=circle] (b5) at  (4,4) {};
\node at (3.5,4) {\Large $1$};

\node [draw, shape=circle] (b6) at  (4,5) {};
\node at (3.5,5) {\Large $1$};

\node [draw, shape=circle] (b61) at  (3,6) {};
\node [draw, shape=circle] (b62) at  (5,6) {};

\node at (4,7) {\Large (b)};

\node [draw, shape=circle] (b21) at  (3,0) {};
\node [draw, shape=circle] (b22) at  (5,0) {};

\draw(b2)--(b3)--(b4)--(b5)--(b6);
\draw(b21)--(b2)--(b22);
\draw(b61)--(b6)--(b62);

\node [draw, shape=circle] (b2) at  (8,1) {};
\node at (7.5,1) {\Large $2$};

\node [draw, shape=circle] (b3) at  (8,2) {};

\node [draw, shape=circle] (b4) at  (8,3) {};
\node at (7.5,3) {\Large $1$};

\node [draw, shape=circle] (b5) at  (8,4) {};

\node [draw, shape=circle] (b6) at  (8,5) {};
\node at (7.5,5) {\Large $2$};

\node [draw, shape=circle] (b61) at  (7,6) {};
\node [draw, shape=circle] (b62) at  (9,6) {};

\node at (8,7) {\Large (c)};

\node [draw, shape=circle] (b21) at  (7,0) {};
\node [draw, shape=circle] (b22) at  (9,0) {};

\draw(b2)--(b3)--(b4)--(b5)--(b6);
\draw(b21)--(b2)--(b22);
\draw(b61)--(b6)--(b62);


\node [draw, shape=circle] (b2) at  (12,1) {};
\node at (11.5,1) {\Large $2$};

\node [draw, shape=circle] (b3) at  (12,2) {};
\node at (11.5,2) {\Large $1$};

\node [draw, shape=circle] (b4) at  (12,3) {};

\node [draw, shape=circle] (b5) at  (12,4) {};
\node at (11.5,4) {\Large $1$};

\node [draw, shape=circle] (b6) at  (12,5) {};
\node at (11.5,5) {\Large $2$};

\node [draw, shape=circle] (b61) at  (11,6) {};
\node [draw, shape=circle] (b62) at  (13,6) {};

\node at (12,7) {\Large (d)};

\node [draw, shape=circle] (b21) at  (11,0) {};
\node [draw, shape=circle] (b22) at  (13,0) {};

\draw(b2)--(b3)--(b4)--(b5)--(b6);
\draw(b21)--(b2)--(b22);
\draw(b61)--(b6)--(b62);


\node [draw, shape=circle] (b2) at  (16,1) {};
\node at (15.5,1) {\Large $2$};

\node [draw, shape=circle] (b3) at  (16,2) {};
\node at (15.5,2) {\Large $1$};

\node [draw, shape=circle] (b4) at  (16,3) {};
\node at (15.5,3) {\Large $1$};

\node [draw, shape=circle] (b5) at  (16,4) {};
\node at (15.5,4) {\Large $1$};

\node [draw, shape=circle] (b6) at  (16,5) {};
\node at (15.5,5) {\Large $2$};

\node [draw, shape=circle] (b61) at  (15,6) {};
\node [draw, shape=circle] (b62) at  (17,6) {};

\node at (16,7) {\Large (e)};

\node [draw, shape=circle] (b21) at  (15,0) {};
\node [draw, shape=circle] (b22) at  (17,0) {};

\draw(b2)--(b3)--(b4)--(b5)--(b6);
\draw(b21)--(b2)--(b22);
\draw(b61)--(b6)--(b62);


\node [draw, shape=circle] (b2) at  (20,1) {};
\node at (19.5,1) {\Large $1$};

\node [draw, shape=circle] (b3) at  (20,2) {};
\node at (19.5,2) {\Large $1$};

\node [draw, shape=circle] (b4) at  (20,3) {};

\node [draw, shape=circle] (b5) at  (20,4) {};
\node at (19.5,4) {\Large $1$};

\node [draw, shape=circle] (b6) at  (20,5) {};
\node at (19.5,5) {\Large $1$};

\node [draw, shape=circle] (b61) at  (19,6) {};
\node [draw, shape=circle] (b62) at  (21,6) {};

\node at (18.5,6) {\Large $1$};
\node at (21.5,6) {\Large $1$};

\node at (20,7) {\Large (f)};

\node [draw, shape=circle] (b21) at  (19,0) {};
\node [draw, shape=circle] (b22) at  (21,0) {};
\node at (19,-0.6) {\Large $1$};
\node at (21,-0.6) {\Large $1$};

\draw(b2)--(b3)--(b4)--(b5)--(b6);
\draw(b21)--(b2)--(b22);
\draw(b61)--(b6)--(b62);

\end{tikzpicture}
\caption{Graph $G$ with different labelings (vertices with no drawn label have label zero) to show the values of several parameters: $\gamma(G)=3$ (a), $\gamma_t(G)=4$ (b), $\gamma_{\{R2\}}(G)=5$ (c), $\gamma_{t\{R2\}}(G)=6$ (d), $\gamma_{tR}(G)=7$ (e) and $\gamma_{\times 2}(G)=8$ (f).}\label{fig1}
\end{figure}

The article is organized as follows. Section $2$ introduces general combinatorial results which show the close relationship that exists between the total Roman $\{2\}$-domination number and other domination parameters. Also, we obtain general bounds and discuss the extreme cases. Finally, in Section $3$ we show that the problem of deciding if a graph has a TR2DF of a given weight is NP-complete, even when restricted to bipartite graphs or chordal graphs.

\subsection{Terminology and Notation}

Given a graph $G$, we denote by $\delta_G(v)=|N_G(v)|$ the \emph{degree} of a vertex $v$ of $G$. Also, $\delta(G) = \min_{v\in V(G)}\{\delta_G(v)\}$ and $\Delta(G) = \max_{v\in V(G)}\{\delta_G(v)\}$. We say that a vertex $v\in V(G) $ is \emph{universal} if $N_G[v]=V(G)$. For a set $S \subseteq V(G)$, its open neighborhood is the set $N_G(S)= \cup_{v\in S} N_G(v)$, and its closed neighborhood is the set $N_G[S]= N_G(S)\cup S$.

The \emph{private neighborhood} $pn_G(v,S)$ of $v \in S\subseteq V(G)$ is defined by $pn_G(v,S)=\{u \in V(G): N_G(u)\cap S = \{v\}\}$. Each vertex in $pn_G(v,S)$ is called a \emph{private neighbor} of $v$ with respect to $S$. The \emph{external private neighborhood} $epn_G(v,S)$ consists of those private neighbors of $v$ in $V(G)\setminus S$. Hence, $epn_G(v,S)= pn_G(v,S)\cap (V(G)\setminus S)$.

For any two vertices $u$ and $v$, the \emph{distance} $d_G(u,v)$ between $u$ and $v$ is the minimum length of a $u-v$ path.  The \emph{diameter} of $G$, denoted by $diam(G)$, is the maximum distance among pairs of vertices of $G$. A \emph{diametral path} in $G$ is a shortest path whose length equals the diameter of the graph. Thus, a diametral path in $G$ is a shortest path joining two vertices that are at distance $diam(G)$ from each other (such vertices are called diametral vertices). From now on, we shall skip the subindex $G$ in all the notations above, whenever the graph $G$ is clear from the context.

Given a set of vertices $S \subseteq V(G)$, by $G-S$ we denote the graph obtained from $G$ by removing all the vertices of $S$ and all the edges incident with a vertex in $S$ (if $S=\{v\}$, for some vertex $v$, then we simply write $G-v$).

A \emph{leaf} vertex of $G$ is a vertex of degree one. A \emph{support} vertex of $G$ is a vertex adjacent to a leaf vertex, a \emph{strong support} vertex is a support vertex adjacent to at least two leaves, a \emph{strong leaf} vertex is a leaf vertex adjacent to a strong support vertex, and a \emph{semi-support} vertex is a vertex adjacent to a support vertex that is not a leaf. The set of leaves is denoted by $L(G)$; the set of support vertices is denoted by $S(G)$; the set of strong support vertices is denoted by $S_s(G)$; the set of strong leaves is denoted by $L_s(G)$; and the set of semi-support vertices is denoted by $SS(G)$.

A \emph{tree} $T$ is an  acyclic connected graph. A \emph{rooted tree} $T$ is a tree with a distinguished special vertex $r$, called the root. For each vertex $v\neq r$ of $T$, the \emph{parent} of $v$ is the neighbor of $v$ on the unique $r-v$ path, while a \emph{child} of $v$ is any other neighbor of $v$. A \emph{descendant} of $v$ is a vertex $u\neq v$ such that the unique $r-u$ path contains $v$. Thus, every child of $v$ is a descendant of $v$. The set of descendants of $v$ is denoted by $D(v)$, and we define $D[v]=D(v)\cup\{v\}$. The \emph{maximal subtree} at $v$ is the subtree of $T$ induced by $D[v]$, and is denoted by $T_v$.

We will use the notation $K_n$, $N_n$, $K_{1,n-1}$, $P_n$ and $C_n$ for complete graphs, empty graphs, star graphs, path graphs and cycle graphs of order $n$, respectively. Given two graphs $G$ and $H$, the \emph{corona product} $G\odot H$ is defined as the graph obtained from $G$ and $H$ by taking one copy of $G$ and $|V(G)|$ copies of $H$, and joining by an edge each vertex of the $i^{th}$-copy of $H$ with the $i^{th}$-vertex of $G$. For the remainder of the article, definitions will be introduced whenever a concept is needed.

\section{Combinatorial results}

We begin this section with two inequality chains relating the domination number, the total domination number, the total Roman domination number, the Roman $\{2\}$-domination number, the double domination number and the total Roman $\{2\}$-domination number. We must remark that the last inequality in the first item is a well known result (see \cite{AbdollahzadehAhangarHenningSamodivkinEtAl2016}). We include it in the result to have a complete vision of the relationship between our parameter and the total domination number.

\begin{proposition}\label{prop-inequalities}
The following inequalities hold for any graph $G$ without isolated vertices.
\begin{enumerate}[{\rm(i)}]
  \item $\gamma_t(G)\leq \gamma_{t\{R2\}}(G)\leq \gamma_{tR}(G)\leq  2\gamma_t(G)$, $(\gamma_{tR}(G)\leq 2\gamma_t(G)$ is from {\em \cite{AbdollahzadehAhangarHenningSamodivkinEtAl2016})}.
  \item$\gamma_{\{R2\}}(G)\leq \gamma_{t\{R2\}}(G)\leq \gamma_{\times 2}(G)$.
\end{enumerate}
\end{proposition}

\begin{proof}
It was shown in  \cite{AbdollahzadehAhangarHenningSamodivkinEtAl2016} that  $\gamma_{tR}(G)\leq 2\gamma_t(G)$. To conclude the proof of (i), we only need to observe that any TR2DF is a TDF, which implies that $\gamma_t(G) \leq \gamma_{t\{R2\}}(G)$, and any TRDF is a TR2DF, which implies that $\gamma_{t\{R2\}}(G)\leq \gamma_{tR}(G)$.

Now, to prove (ii), we only need to observe that any DDF is a TR2DF, which implies that $\gamma_{t\{R2\}}(G)\leq \gamma_{\times 2}(G)$ and any TR2DF is a R2DF, which implies that $\gamma_{\{R2\}}(G)\leq \gamma_{t\{R2\}}(G)$.
\end{proof}

The following result provides equivalent conditions for the graphs where the left hand side inequality of Proposition \ref{prop-inequalities} (i) is achieved. Note that it has a simple proof, but we however prefer include it to have a more complete exposition.

\begin{remark}\label{equiv-tR2-x2}
For any graph $G$, the following statements are equivalent.
   \begin{enumerate}[{\rm(a)}]
     \item $\gamma_{t\{R2\}}(G)=\gamma_t(G)$.
     \item $\gamma_{\times 2}(G)=\gamma_t(G)$.
   \end{enumerate}
\end{remark}

\begin{proof}
Suppose that (a) holds and let $f(V_0, V_1,V_2)$ be a $\gamma_{t\{R2\}}(G)$-function. Since $f$ is a TDF, $\gamma_t(G)\leq |V_1\cup V_2|=|V_1|+|V_2|\leq |V_1|+2|V_2|=\gamma_{t\{R2\}}(G)=\gamma_t(G)$. So $V_2=\emptyset$, which implies that $f$ is a DDF of weight $\omega(f)=\gamma_t(G)$. Hence, (b) holds. Finally, it is straightforward to observe that (b) implies (a).
\end{proof}

We continue by showing a simple relationship between the total Roman $\{2\}$-domination number, the domination number and the total domination number. Since $\gamma(G)\le \gamma_t(G)$ for any graph $G$, we notice that the following result improves the last upper bound of Proposition \ref{prop-inequalities}~(i).

\begin{theorem}\label{teo-t-g}
For any graph $G$ without isolated vertices,
$\gamma_{t\{R2\}}(G)\leq \gamma_t(G)+\gamma(G).$
\end{theorem}

\begin{proof}
Let $D$ be a $\gamma_t(G)$-set and let $S$ be a $\gamma(G)$-set. We define the function $f(V_0,V_1,V_2)$ on $G$, where $V_2=D\cap S$ and $V_1=(D\cup S)\setminus V_2$. Notice that $f$ is a TR2DF on $G$ of weight $\omega(f)=2|V_2|+|V_1|=|D|+|S|= \gamma_t(G)+\gamma(G)$. Therefore, the result follows.
\end{proof}

The following result is an immediate consequence of the remark above and the well-know inequality $\gamma_t(G)\leq 2\gamma(G)$ (see \cite{Henning2009}).

\begin{corollary}\label{cor-t-g}
For any graph $G$ without isolated vertices,
$\gamma_{t\{R2\}}(G)\leq 3\gamma(G).$
\end{corollary}

We remark that the upper bound of Theorem \ref{teo-t-g} is sharp. For example, for an integer $s\geq 1$, let $H_s$ be the graph obtained from $P_3$ and $N_1$ by taking one copy of $N_1$ and $s$ copies of $P_3$, and joining by an edge the support vertex of each copy of $P_3$ with the vertex of $N_1$. It is easy to check that $\gamma(H_s)=s$, $\gamma_t(H_s)=s+1$ and $\gamma_{t\{R2\}}(H_s)=2s+1=\gamma_t(H_s)+\gamma(H_s)$. The graph $H_3$, for example, is illustrated in Figure \ref{figure-H-s}.

\begin{figure}[h]
\centering
\begin{tikzpicture}[scale=.6, transform shape]

\node [draw, shape=circle] (u) at  (0,1.5) {};

\node [draw, shape=circle] (s1) at  (-3,0) {};
\node [draw, shape=circle] (s2) at  (0,0) {};
\node [draw, shape=circle] (s3) at  (3,0) {};

\node [draw, shape=circle] (h11) at  (-3.5,-1.5) {};
\node [draw, shape=circle] (h12) at  (-2.5,-1.5) {};

\node [draw, shape=circle] (h21) at  (-0.5,-1.5) {};
\node [draw, shape=circle] (h22) at  (0.5,-1.5) {};

\node [draw, shape=circle] (h31) at  (2.5,-1.5) {};
\node [draw, shape=circle] (h32) at  (3.5,-1.5) {};

\draw(u)--(s1);
\draw(u)--(s2);
\draw(u)--(s3);

\draw(s1)--(h11);
\draw(s1)--(h12);

\draw(s2)--(h21);
\draw(s2)--(h22);

\draw(s3)--(h31);
\draw(s3)--(h32);

\end{tikzpicture}
\caption{The graph $H_3$.}
\label{figure-H-s}
\end{figure}
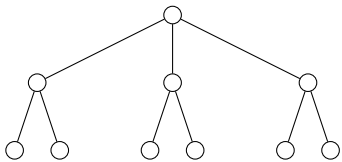

From Proposition \ref{prop-inequalities} and Theorem \ref{teo-t-g}, we immediately obtain that $\gamma_t(G)=\gamma(G)$ is a necessary condition for a graph $G$ to satisfy $\gamma_{t\{R2\}}(G)=2\gamma_t(G)$. However, this condition is not sufficient, for example, the cycle graph $C_4$ satisfies that $\gamma_t(C_4)=\gamma(C_4)=2$ and $\gamma_{t\{R2\}}(C_4)=3<4=2\gamma_t(C_4)$.

The following result provides an equivalent condition for the graphs $G$ which satisfy the equality $\gamma_{t\{R2\}}(G)=2\gamma_t(G)$. Before we shall need the following known result.

\begin{theorem}\label{teo-tR-2g}{\rm\cite{AbdollahzadehAhangarHenningSamodivkinEtAl2016}}
If $G$ is a graph with no isolated vertex, then $2\gamma(G)\leq \gamma_{tR}(G)$.
\end{theorem}

\begin{theorem}\label{teo-tR2-2t}
Let $G$ be a graph. Then $\gamma_{t\{R2\}}(G)=2\gamma_t(G)$ if and only if $\gamma_{t\{R2\}}(G)=\gamma_{tR}(G)$ and $\gamma_t(G)=\gamma(G)$.
\end{theorem}

\begin{proof}
Assume that $\gamma_{t\{R2\}}(G)=2\gamma_t(G)$. Hence, Proposition \ref{prop-inequalities} leads to $\gamma_{t\{R2\}}(G)=\gamma_{tR}(G)$. Also, by Theorem \ref{teo-t-g} and the known inequality $\gamma(G)\leq \gamma_t(G)$, we obtain that $2\gamma_t(G)=\gamma_{t\{R2\}}(G)\leq \gamma_t(G)+\gamma(G)\leq 2\gamma_t(G)$. Therefore, we must have equality throughout the inequality chain above. In particular, $\gamma_t(G)=\gamma(G)$.

On the other hand, we assume that $\gamma_{t\{R2\}}(G)=\gamma_{tR}(G)$ and $\gamma_t(G)=\gamma(G)$. By the equalities above, Theorem \ref{teo-tR-2g} and Proposition \ref{prop-inequalities}, we obtain that $2\gamma_t(G)=2\gamma(G)\leq \gamma_{tR}(G)=\gamma_{t\{R2\}}(G)\leq 2\gamma_t(G)$. Therefore $\gamma_{t\{R2\}}(G)= 2\gamma_t(G)$.
\end{proof}

Notice that the inequality $\gamma_{t\{R2\}}(G)\leq 3\gamma(G)$  can be also deduced from the following result.

\begin{theorem}\label{theom-tr2-r2}
For any graph $G$ without isolated vertices, $\gamma_{t\{R2\}}(G)\leq \gamma_{\{R2\}}(G)+\gamma(G).$
\end{theorem}

\begin{proof}
Let $f(V_0,V_1,V_2)$ be a $\gamma_{\{R2\}}(G)$-function and let $S$ be a $\gamma(G)$-set. Now, we consider the function $f'(V_0',V_1',V_2')$ defined as follows.
\begin{enumerate}[{\rm (a)}]
  \item For every $x\in (V_1\cup V_2)\cap S$, choose a vertex $u\in (V_0\cap N(x))\setminus S$ if it exists, and label it as $f'(u)=1$.
  \item  For every vertex $x\in V_0\cap S$, $f'(x)=1$.
  \item  For any other vertex $u$ not previously labelled, $f'(u)=f(u)$.
\end{enumerate}
Since $f$ is a R2DF, by definition, $f'$ is a R2DF as well.
Observe that $f'$ is also a TDF on $G$. Thus, $f'$ is a TR2DF on $G$, and therefore, $\gamma_{t\{R2\}}(G)\leq \omega(f')\leq \gamma_{\{R2\}}(G)+\gamma(G)$.
\end{proof}

The bound above is tight. For instance, it is achieved for the star graph $K_{1,n-1}$, where $n\geq 3$.

\begin{corollary}
For any graph $G$ without isolated vertices, $\gamma_{t\{R2\}}(G)\leq 2\gamma_{\{R2\}}(G)$. Furthermore, if  $\gamma_{\{R2\}}(G)>\gamma(G)$,
then $\gamma_{t\{R2\}}(G)\leq 2\gamma_{\{R2\}}(G)-1$.
\end{corollary}

In connection with the sharpness of the latter bound of the corollary above, we observe that every graph $G$ having exactly one universal vertex satisfies that $\gamma_{t\{R2\}}(G)= 2\gamma_{\{R2\}}(G)-1$.

The next result establishes the existence of a $\gamma_{t\{R2\}}(G)$-function which satisfies a useful property.

\begin{proposition}\label{prop_private}
For any graph $G$ without isolated vertices, there exists a $\gamma_{t\{R2\}}(G)$-function $f(V_0,V_1,V_2)$ such that either $V_2=\emptyset$ or every vertex of $V_2$ has at least two private neighbors in $V_0$ with respect to the set $V_1 \cup V_2$.
\end{proposition}

\begin{proof}
Let $f(V_0,V_1,V_2)$ be a $\gamma_{t\{R2\}}(G)$-function satisfying that $|V_2|$ is minimum. Clearly, if $|V_2|=0$, then we are done. Hence, let $v\in V_2$. If $epn(v, V_1\cup V_2)=\emptyset$, then the function $f'$, defined by $f'(v)=1$ and $f'(x)=f(x)$ whenever $x\in V(G)\setminus \{v\}$, is a TR2DF on $G$, which is a contradiction, and so, $epn(v, V_1\cup V_2)\neq\emptyset$. If $epn(v, V_1\cup V_2)=\{u\}$, then the function $f''$, defined by $f''(v)=f''(u)=1$ and $f''(x)=f(x)$ whenever $x\in V(G)\setminus \{v,u\}$, is a TR2DF on $G$, which is a contradiction as well. Thus, $|epn(v, V_1\cup V_2)|\geq 2$, which completes the proof.
\end{proof}

\begin{corollary}\label{cor-Pn-Cn}
For every graph $G$ without isolated vertices and maximum degree $\Delta(G) \leq 2$, $$\gamma_{t\{R2\}}(G)=\gamma_{\times2}(G).$$
\end{corollary}

From Corollary \ref{cor-Pn-Cn}, and the following values of $\gamma_{\times 2}(P_n)$ and $\gamma_{\times 2}(C_n)$ obtained in \cite{chellali2006} and \cite{haynes2000} respectively, we obtain our next result.

$$\gamma_{\times 2}(P_n)=\left\{\begin{array}{ll}
                                 2\left\lceil \frac{n}{3}\right\rceil + 1, & \mbox{if $n\equiv 0$ $($mod $3)$,} \\[.3cm]
                                 2\left\lceil \frac{n}{3}\right\rceil, & \mbox{otherwise.}
                               \end{array}
\right.\mbox{ and }\;\;\gamma_{\times 2}(C_n)=\left\lceil \frac{2n}{3}\right\rceil.$$

\begin{remark}\label{rem-Cn-Pn}
For any positive integer $n\geq 2$,

\begin{enumerate}[{\rm (i)}]
\item $\gamma_{t\{R2\}}(P_n)=\left\{\begin{array}{ll}
                                 2\left\lceil \frac{n}{3}\right\rceil + 1, & \mbox{if $n\equiv 0$ $($mod $3)$,} \\[.3cm]
                                 2\left\lceil \frac{n}{3}\right\rceil, & \mbox{otherwise.}
                               \end{array}
\right.$

\item $\gamma_{t\{R2\}}(C_n)=\left\lceil \frac{2n}{3}\right\rceil.$
\end{enumerate}
\end{remark}

Our next contribution shows another relationship between our parameter and the total domination number, but we now also use the order of the graph.

\begin{theorem}\label{teo-total-n}
For any graph $G$ of order $n$ and $\delta(G)\geq 2$,  $$\gamma_{t\{R2\}}(G)\leq \left\lfloor\frac{\gamma_t(G)+n}{2}\right\rfloor.$$
\end{theorem}

\begin{proof}
Let $D$ be a $\gamma_t(G)$-set, let $I$ be the set of isolated vertices in $\langle V(G)\setminus D\rangle$ and let $S$ be a $\gamma(\langle V(G)\setminus (D\cup I)\rangle)$-set. In addition, let $f(V_0,V_1,\emptyset)$ be a function defined by $V_1=D\cup S$ and $V_0=V(G)\setminus V_1$. Since $D$ is a TDS of $G$, we have that $V_1=D\cup S$ is a TDS as well. Furthermore,  every vertex  $u\in V(G)\setminus (D\cup S)$ is dominated  by at least two vertices of $ V_1$. Hence, $V_1$ is a double dominating set of $G$, which implies that $f$ is a TR2DF on $G$. Thus, $\gamma_{t\{R2\}}(G)\leq |V_1|=|D\cup S|=|D|+|S|$. Now, since $\langle V(G)\setminus (D\cup I) \rangle$ is a graph without isolated vertices, we have that $|S|=\gamma(\langle V(G)\setminus (D\cup I)\rangle)\leq \frac{|V(G)\setminus (D\cup I)|}{2}\leq \frac{|V(G)\setminus D|}{2}=\frac{n-\gamma_t(G)}{2}$. Therefore,
$\gamma_{t\{R2\}}(G)\leq  \lfloor\frac{\gamma_t(G)+n}{2}\rfloor$, which completes the proof.
\end{proof}

To see the tightness of the bound above we consider for instance the Cartesian product graph $P_2\square P_3$. Also, a consequence of such theorem above is next stated. This is also based on the fact that for any graph $G$ with $\delta(G)\geq 3$, $\gamma_t(G)\leq \frac{|V(G)|}{2}$.

\begin{proposition}\label{teo-total-n-consequence}
For any graph $G$ of order $n$ and $\delta(G)\geq 3$,  $$\gamma_{t\{R2\}}(G)\leq \frac{3n}{4}.$$
\end{proposition}

Given a graph $G$ and an edge $e\in E(G)$, the graph obtained from $G$ by removing the edge $e$ will be denoted by $G-e$. Notice that any $\gamma_{t\{R2\}}(G-e)$-function is a TR2DF on $G$. Therefore, the following basic result follows.

\begin{observation}\label{obs-span-subgraph}
If $H$ is a spanning subgraph $($without isolated vertices$)$ of a graph $G$, then $\gamma_{t\{R2\}}(G)\leq \gamma_{t\{R2\}}(H)$.
\end{observation}

From Remark \ref{rem-Cn-Pn} and Observation \ref{obs-span-subgraph}, we obtain the following result.

\begin{proposition}
Let $G$ be a graph of order $n$.

\begin{enumerate}
\item[-] If $G$ is a Hamiltonian graph, then  $\gamma_{t\{R2\}}(G)\leq 2\left\lceil \frac{n}{3}\right\rceil$.
\item[-] If $G$ has a Hamiltonian path, then $\gamma_{t\{R2\}}(G)\leq 2\left\lceil \frac{n}{3}\right\rceil + 1$.
\end{enumerate}
\end{proposition}

Clearly, the bounds above are tight, as they are achieved for $C_n$ and $P_n$ with $n\equiv 0$ $(mod$ $3)$, respectively.

We now proceed to characterize all graphs achieving the limit cases of the trivial bounds $2\leq \gamma_{t\{R2\}}(G)\leq n$. For this purpose, we shall need the following theorem.

\begin{theorem}{\rm\cite{haynes2000}}\label{teo-haynes-n}
Let $G$ be a graph without isolated vertices. Then $\gamma_{\times 2}(G)=2$ if and only if $G$ has two universal vertices.
\end{theorem}

\begin{theorem}\label{teo-tR2=2}
Let $G$ be a graph without isolated vertices. Then $\gamma_{t\{R2\} }(G)=2$ if and only if $G$ has two universal vertices.
\end{theorem}

\begin{proof}
Notice that $\gamma_{t\{R2\}}(G)=2$ directly implies $\gamma_{\times 2}(G)=2$. Hence, by Theorem \ref{teo-haynes-n}, $G$ has two universal vertices. The other hand it is straightforward to see.
\end{proof}

We next proceed to characterize all graphs $G$ with $\gamma_{t\{R2\}}(G)=3$. For this purpose, we consider the next family of graphs. Let $\mathcal{H}$ be the family of graphs $H$ of order $n\geq 3$ such that the subgraph induced by three vertices of $H$ is $P_3$ or $C_3$ and the remaining $n-3$ vertices have minimum degree two and they induce an empty graph.

\begin{theorem}\label{teo-tR2=3}
Let $G$ be a connected graph of order $n$. Then $\gamma_{t\{R2\}}(G)=3$ if and only if there exists $H\in \mathcal{H}\cup \{K_{1,n-1}\}$ which is a spanning subgraph of $G$ and $G$ has as most one universal vertex.
\end{theorem}

\begin{proof}
We first suppose that $\gamma_{t\{R2\}}(G)=3$. Let $f(V_0, V_1, V_2)$ be a $\gamma_{t\{R2\}}(G)$-function. By Theorem \ref{teo-tR2=2},
$G$ has at most one universal vertex. If $|V_2|=1$, then $|V_1|=1$. Let $V_1=\{v\}$ and $V_2 = \{w\}$. Notice that $v$ and $w$ are adjacent vertices. Since $f$ is a TR2DF, any vertex must be adjacent to $w$, concluding that $K_{1,n-1}$ is a spanning subgraph of $G$. Now, if $|V_2| = 0$, then $|V_1| = 3$. As $V_1$ is a TDS, the subgraph induced by $V_1$ is $P_3$ or $C_3$. Since $f$ is a TR2DF, we observe that $|N(x) \cap V_1| \geq 2$ for every $x\in V_0$. Hence, in this case, $G$ contains a spanning subgraph belonging to $\mathcal{H}$.

Conversely, let $G$ be a connected graph of order $n$ containing a graph $H\in \mathcal{H}\cup \{K_{1,n-1}\}$ as  a spanning subgraph. Notice that we can construct a TR2DF $g$ satisfying that $\omega(g)=3$. Hence $\gamma_{t\{R2\}}(G)\leq \omega(g)=3$. Moreover, since $G$ has at most one universal vertex, by Theorem \ref{teo-tR2=2} we have that $\gamma_{t\{R2\}}(G)\geq 3$, which completes the proof.
\end{proof}

\begin{theorem}\label{teo-tR2=n}
Let $G$ be a connected graph of order $n$. Then $\gamma_{t\{R2\}}(G)=n$ if and only if $G$ is $P_3$ or $H\odot N_1$ for some connected graph $H$.
\end{theorem}

\begin{proof}
If $G$ is $P_3$ or $H\odot N_1$ for some connected graph $H$, then it is straightforward to see that $\gamma_{t\{R2\}}(G)=n$. From now on we assume that $G$ is a connected graph such that $\gamma_{t\{R2\}}(G)=n$. If $n=2$, then $G\cong P_2\cong N_1\odot N_1$, and if $n=3$, then $G\cong P_3$. Hence, we consider that $n\geq 4$. Suppose there exists a vertex $v\notin L(G)\cup S(G)$. Notice that the function $f$, defined by $f(v)=0$ and $f(x)=1$ whenever $x\in V(G)\setminus \{v\}$, is a TR2DF of weight $\omega(f)=n-1$, which is a contradiction. Thus $V(G)=L(G)\cup S(G)$.

Now, suppose there exists a vertex $u\in S_s(G)$ and let $h_1,h_2$ be two leaves adjacent to $u$. We consider the function $g$ defined by $g(h_1)=g(h_2)=0$, $g(u)=2$ and $g(x)=1$ whenever $x\in V(G)\setminus \{u,h_1,h_2\}$. Hence, $g$ is a TR2DF of weight $\omega(g)=n-1$, which is again a contradiction. Thus $S_s(G)=\emptyset$ and, as a consequence, $G\cong H\odot N_1$ for some connected graph $H$.
\end{proof}

Based on the trivial bound $2\leq \gamma_{t\{R2\}}(G)\leq n$ and the characterizations above, it is natural to think into the existence of graphs achieving all the other possible values in the range given by such bounds for $\gamma_{t\{R2\}}(G)$. That is made in our next result, and for it, we need two previous observations that appear first.

\begin{observation}\label{obs-sup-adj}
For any connected graph $G$ containing two adjacent support vertices $v$ and $w$, there exists a $\gamma_{t\{R2\}}(G)$-function $f$ satisfying $f(v)=f(w)=2$.
\end{observation}

\begin{observation}\label{obs-strong}
Let $G$ be a connected graph different from a star graph. If $v\in S_s(G)$, then there exists a $\gamma_{t\{R2\}}(G)$-function ($\gamma_{tR}(G)$-function) $f$ satisfying that $f(v)=2$ and $f(N(v)\cap L(G))=0$.
\end{observation}

\begin{proposition}
For any integers $r,n$ with $3<r<n$, there exists a graph $F_{r,n}$ of order $n$ such that $\gamma_{t\{R2\}}(F_{r,n})=r$.
\end{proposition}

\begin{proof}
If $r$ is even, then we consider a graph $F_{r,n}$ constructed as follows. We begin with a corona product graph $H\odot N_1$ of order $|V(H\odot N_1)|=r$ and $n-r$ isolated vertices. To obtain $F_{r,n}$, we join (by an edge) one vertex $v$ of $H$ to each one of the $n-r$ isolated vertices. Notice that $F_{r,n}$ has order $n$. By Observation \ref{obs-sup-adj}, the function $f$, defined by $f(x)=2$ if $x\in V(H)$ and $f(x)=0$ if $x\in V(F_{r,n})\setminus V(H)$, is a $\gamma_{t\{R2\}}(F_{r,n})$-function and so, $\gamma_{t\{R2\}}(F_{r,n})=\omega(f)=r$.

On the other hand, if $r$ is odd, we construct a graph  $F_{r,n}$ as follows. We begin with a corona product graph $H\odot N_1$ of order $|V(H\odot N_1)|=r-3$ and a star graph $K_{1,n-r+2}$. To obtain $F_{r,n}$, we join (by an edge) one vertex $v$ of $H$ to one leaf, namely $h$, of the star $K_{1,n-r+2}$. Hence, $F_{r,n}$ has order $n$. Now, we consider the function $f$, defined by $f(h)=1$, $f(x)=2$ if $x\in S(F_{r,n})$ and $f(x)=0$ otherwise. Notice that $f$ is a TR2DF  on $F_{r,n}$ and so, $\gamma_{t\{R2\}}(F_{r,n})\leq \omega(f)=r$.

Let $v$ be the support vertex of $K_{1,n-r+2}$. Since $|V(K_{1,n-r+2})|\geq 4$, $v\in S_s(F_{r,n})$. By Observation \ref{obs-strong}, there exists a $\gamma_{t\{R2\}}(F_{r,n})$-function $g(V_0,V_1,V_2)$ such that $g(v)=2$ and $g(x)=0$ if $x\in N(v)\cap L(F_{r,n})$. Hence $g(h)\geq 1$ because $V_1\cup V_2$ is a TDS of $F_{r,n}$. Moreover, notice that the function $g$ restricted to $V(H\odot N_1)$, say $g'$, is a TR2DF on $H\odot N_1$. So, by statement above and Theorem \ref{teo-tR2=n}, $\omega(g')\geq \gamma_{t\{R2\}}(H\odot N_1)=r-3$. Therefore,
$\gamma_{t\{R2\}}(F_{r,n})=\omega(g)=g(N[v])+\omega(g')\geq 3+r-3=r$. Consequently, it follows that $\gamma_{t\{R2\}}(F_{r,n})=r$ and the proof is complete.
\end{proof}

\subsection{Trees $T$ with $\gamma_{t\{R2\}}(T)=\gamma_{tR}(T)$}

We begin this subsection with a theoretical characterization of the graphs $G$ satisfying the equality $\gamma_{t\{R2\}}(G)=\gamma_{tR}(G)$.

\begin{theorem}\label{teo-tR2-tR}
Let $G$ be a graph. Then $\gamma_{t\{R2\}}(G)=\gamma_{tR}(G)$ if and only if there exists a $\gamma_{t\{R2\}}(G)$-function $f(V_0,V_1,V_2)$ such that $V_{0,1}=\emptyset$.
\end{theorem}

\begin{proof}
Suppose that $\gamma_{t\{R2\}}(G)=\gamma_{tR}(G)$. Let $f(V_0,V_1,V_2)$ be a $\gamma_{tR}(G)$-function. Since every TRDF is a TR2DF, $f$ is a $\gamma_{t\{R2\}}(G)$-function as well, and satisfies that $V_{0,1}=\emptyset$. Conversely, suppose there exists a $\gamma_{t\{R2\}}(G)$-function $f'(V_0',V_1',V_2')$ such that $V_{0,1}'=\emptyset$. So, $V_0'=V_{0,2}'$, which implies that $f'$ is a TRDF on $G$. Thus, $\gamma_{tR}(G)\leq \omega(f')=\gamma_{t\{R2\}}(G)$. Hence, Proposition \ref{prop-inequalities} leads to $\gamma_{t\{R2\}}(G)=\gamma_{tR}(G)$.
\end{proof}

The characterization above clearly lacks of usefulness since it precisely depends on finding a $\gamma_{t\{R2\}}(G)$-function which satisfies a specific condition. In that sense, it appears an open problem to characterize the graphs $G$ which satisfy the equality $\gamma_{t\{R2\}}(G)=\gamma_{tR}(G)$. In this subsection we give a partial solution to this problem for the particular case of trees. To this end, we require the next results and extra definitions.

\begin{observation}\label{obs-strong-var-1}
Let $G$ be a connected graph. If $v\in S_s(G)$, then there exists a $\gamma_{t\{R2\}}(G)$-function ($\gamma_{tR}(G)$-function) $f$ satisfying that $f(v)=2$ and $f(h)=0$ for some vertex $h\in N(v)\cap L(G)$.
\end{observation}

\begin{observation}\label{obs-lem-tR2-tR}
If $T'$ is a subtree of a tree $T$, then $\gamma_{t\{R2\}}(T')\leq \gamma_{t\{R2\}}(T)$ and $\gamma_{tR}(T')\leq \gamma_{tR}(T)$.
\end{observation}

By an \emph{isolated support vertex} of $G$ we mean an isolated vertex of the subgraph induced by the support vertices of $G$. The set of non isolated support vertices of $G$ is denoted by $S_{adj}(G)$.

The set of support vertices of $G$ labelled with two by some $\gamma_{tR}(G)$-function is denoted by $S_{tR,2}(G)$. The set of leaves of $G$ labelled with one by some $\gamma_{tR}(G)$-function is denoted by $L_{tR,1}(G)$.
The set of vertices of $G$ labelled with zero by all $\gamma_{t\{R2\}}(G)$-functions is denoted by $W_0(G)$.
The set of support vertices of $G$ labelled with one by all $\gamma_{t\{R2\}}(G)$-functions is denoted by $S_1(G)$.

For an integer $r\geq 1$, the graph $R_r$ is defined as the graph obtained from $P_4$ and $N_1$ by taking one copy of $N_1$ and $r$ copies of $P_4$ and joining by an edge one support vertex of each copy of $P_4$ with the vertex of $N_1$. In Figure \ref{figure-1} we show the example of $R_3$.

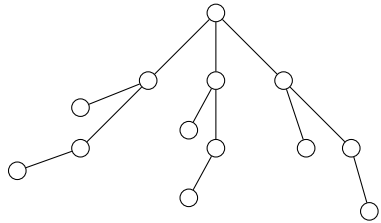
\begin{figure}[h]
\centering
\begin{tikzpicture}[scale=.6, transform shape]

\node [draw, shape=circle] (u) at  (0,1.5) {};

\node [draw, shape=circle] (s1) at  (-1.5,0) {};
\node [draw, shape=circle] (s2) at  (0,0) {};
\node [draw, shape=circle] (s3) at  (1.5,0) {};

\node [draw, shape=circle] (h1s1) at  (-3,-0.6) {};

\node [draw, shape=circle] (h1s2) at  (-0.6,-1.1) {};

\node [draw, shape=circle] (h2s3) at  (2,-1.5) {};

\node [draw, shape=circle] (s11) at  (-3,-1.5) {};
\node [draw, shape=circle] (s22) at  (0,-1.5) {};
\node [draw, shape=circle] (s33) at  (3,-1.5) {};

\node [draw, shape=circle] (h1s11) at  (-4.4,-2) {};

\node [draw, shape=circle] (h1s22) at  (-0.6,-2.6) {};

\node [draw, shape=circle] (h2s33) at  (3.4,-2.9) {};

\draw(u)--(s1);
\draw(u)--(s2);
\draw(u)--(s3);

\draw(s1)--(h1s1);
\draw(s1)--(s11);

\draw(s2)--(h1s2);
\draw(s2)--(s22);

\draw(s3)--(h2s3);
\draw(s3)--(s33);

\draw(s11)--(h1s11);

\draw(s22)--(h1s22);

\draw(s33)--(h2s33);

\end{tikzpicture}
\caption{The structure of the tree $R_3$.}
\label{figure-1}
\end{figure}

A \emph{near total Roman $\{2\}$-dominating function relative to a vertex $v$}, abbreviated near-TR2DF relative to $v$, on a graph $G$, is a function $f(V_0,V_1,V_2)$ satisfying the following.

\begin{itemize}
\item For each vertex $u\in V_0$, if $u=v$, then $\sum_{u\in N(v)}f(u)\geq 1$, while if $u\neq v$, then $\sum_{u\in N(v)}f(u)\geq 2$.
\item The subgraph induced by $V_1\cup V_2$ has no isolated vertex.
\end{itemize}

The weight of a near-TR2DF relative to $v$ on $G$ is the value $f(V(G))=\sum_{u\in V(G)} f(u)$. The minimum weight of a near-TR2DF relative to $v$ on $G$ is called the \emph{near total Roman $\{2\}$-domination number relative to $v$ of $G$}, which we denote as $\gamma_{t\{R2\}}^n(G;v)$. Since every TR2DF is a near-TR2DF, we note that $\gamma_{t\{R2\}}^n(G;v)\leq \gamma_{t\{R2\}}(G)$ for any vertex $v$ of $G$. We define a vertex $v\in V(G)$ to be a \emph{near stable vertex} of $G$ if $\gamma_{t\{R2\}}^n(G;v)=\gamma_{t\{R2\}}(G)$. For example, every leaf of any star $K_{1,n-1}$ with $n\geq 4$, is a near stable vertex. We remark that the terminology of ``near'' style parameters is a commonly used technique in domination theory. In order to simply  mention a recently published example where this was used, we can for instance refer to \cite{henning2016}.

Now on, in order to provide a constructive characterization for the trees which achieve the stated equality in Theorem \ref{teo-tR2-tR}, we consider the next family of trees. Let $\mathcal{F}$ be the family of trees $T$ that can be obtained from a sequence of trees $T_0, \ldots , T_k$, where $k \geq 0$, $T_0 \cong P_2$ and $T\cong T_k$. Furthermore, if $k \geq 1$, then for each $i \in \{1, \ldots ,k\}$, the tree $T_i$ can be obtained from the tree $T'\cong T_{i-1}$ by one of the following operations $F_1$, $F_2$, $F_3$, $F_4$,  $F_5$, $F_6$ or $F_7$. In such operations, by a join of two vertices we mean adding an edge between these two vertices.

\begin{description}
  \item[Operation $F_1$:] Add a tree $R_r$ with semi-support vertex $u$, and join $u$ to an arbitrary vertex $v$ of $T'$.
  \item[Operation $F_2$:]  Add a new vertex $u$ to $T'$ and join $u$ to a vertex $v\in S_{tR,2}(T')$.
  \item[Operation $F_3$:]  Add a new vertex $u$ to $T'$ and join $u$ to a vertex $v\in S_1(T')$.
  \item[Operation $F_4$:] Add a path $P_2$ and join a leaf to a  vertex $v\in S_{adj}(T')$.
  \item[Operation $F_5$:] Add a path $P_3$ with support vertex $u$,  and identify $u$ with a vertex $v\in L_{tR,1}(T')$.
  \item[Operation $F_6$:] Add a path $P_2$, and join a leaf to a near stable vertex $v\in L(T')\cup SS(T')$.
  \item[Operation $F_7$:] Add a path $P_3$, and join a leaf to a vertex $v\in W_0(T')$.
  \end{description}

We next show that every tree $T$ in the family $\mathcal{F}$ satisfies that $\gamma_{t\{R2\}}(T)=\gamma_{tR}(T)$.

\begin{theorem}\label{lem-right-1}
If $T \in \mathcal{F}$, then $\gamma_{t\{R2\}}(T)=\gamma_{tR}(T)$.
\end{theorem}

\begin{proof}
We proceed by induction on the number $r(T)$ of operations required to construct the tree $T$. If $r(T)=0$, then $T\cong P_2$ and satisfies that $\gamma_{t\{R2\}}(T)=2=\gamma_{tR}(T)$. This establishes the base case. Hence, we now  assume that $k\geq 1$ is an integer and that each tree $T' \in \mathcal{F}$ with $r(T')<k$ satisfies that $\gamma_{t\{R2\}}(T')=\gamma_{tR}(T')$. Let $T \in \mathcal{F}$ be a tree with $r(T)=k$. Then, $T$ can be obtained from a tree $T' \in \mathcal{F}$ with $ r(T')=k-1$ by one of the seven operations above. We shall prove that $T$ satisfies that $\gamma_{t\{R2\}}(T)=\gamma_{tR}(T)$. We consider seven cases, depending on which operation is used to construct the tree $T$ from $T'$.\\

\noindent
\textbf{Case 1.} $T$ is obtained from $T'$ by Operation $F_1$. Assume $T$ is obtained from $T'$ by adding a tree $R_r$, being $u$ the semi-support vertex, and the edge $uv$ where $v$ is an arbitrary vertex of $T'$. Observe that, from any TRDF on $T'$, we can obtain a TRDF on $T$ by assigning the weight two to each support vertex and zero to another vertices of $R_r$. Hence, by Proposition \ref{prop-inequalities}, statement above and inductive hypothesis, we obtain
\begin{equation}\label{eq1}
\gamma_{t\{R2\}}(T)\leq \gamma_{tR}(T)\leq \gamma_{tR}(T')+4r=\gamma_{t\{R2\}}(T')+4r.
\end{equation}

Since $S(R_r)=S_{adj}(R_r)$, by using Observation \ref{obs-sup-adj}, there exists a $\gamma_{t\{R2\}}(T)$-function $f$ satisfying that $f(x)=2$ for every $x\in S(R_r)$. As a consequence, $f(h)=0$ for every vertex $h\in L(R_r)$.

If $f(u)=0$, then $f$ restricted to $V(T')$ is a TR2DF on $T'$,  implying that $\gamma_{t\{R2\}}(T')\leq f(V(T'))=\omega(f)-f(V(R_r))=\gamma_{t\{R2\}}(T)-4r$, and by the inequality chain (\ref{eq1}) it follows that $\gamma_{t\{R2\}}(T)=\gamma_{tR}(T)$.

Let $w\in N(v)\setminus \{u\}$ be a vertex such that $f(w)=\max\{f(x): x\in N(v)\setminus \{u\}\}$. If $f(u)>0$, then the function $g$, defined by $g(v)=\max\{f(v),f(v)f(w)+1\}$, (note that in any possibility, this maximum expression can never take a value larger than two), $g(w)=\max\{1,f(w)\}$ $g(u)=0$ and $g(x)=f(x)$ whenever $x\in V(T)\setminus \{v,w,u\}$, is a TR2DF on $T$ with weight $\omega(g)=\omega(f)=\gamma_{t\{R2\}}(T)$. So, $g$ is a $\gamma_{t\{R2\}}(T)$-function as well. As $g(u)=0$, $g$ restricted to $V(T')$ is a TR2DF on $T'$ and, by using a similar reasoning as in the previous case ($f(u)=0$), we obtain that $\gamma_{t\{R2\}}(T)=\gamma_{tR}(T)$.\\

\noindent
\textbf{Case 2.} $T$ is obtained from $T'$ by Operation $F_2$.
Assume $T$ is obtained from $T'$ by adding a new vertex $u$ and the edge $uv$, where $v\in S_{tR,2}(T')$. Hence, there exists a $\gamma_{tR}(T')$-function $f$ satisfying that $f(v)=2$. Notice that $f$ can be extended to a TRDF on $T$ by assigning the weight $0$ to $u$, which implies that $\gamma_{tR}(T)\leq \gamma_{tR}(T')$. By using Proposition \ref{prop-inequalities}, inequality above, inductive hypothesis and Observation \ref{obs-lem-tR2-tR}, we obtain $\gamma_{t\{R2\}}(T)\leq \gamma_{tR}(T)\leq \gamma_{tR}(T')=\gamma_{t\{R2\}}(T')\leq \gamma_{t\{R2\}}(T)$. Therefore, we must have equality throughout the inequality chain above. In particular, $\gamma_{t\{R2\}}(T)=\gamma_{tR}(T)$.\\

\noindent
\textbf{Case 3.} $T$ is obtained from $T'$ by Operation $F_3$.
Assume $T$ is obtained from $T'$ by adding a new vertex $u$ and the edge $uv$, where $v\in S_1(T')$. Since $v$ is a support of $T'$, $v$ is a strong support of $T$. So, by Observation \ref{obs-strong-var-1}, there exists a $\gamma_{t\{R2\}}(T)$-function $g$ satisfying that $g(v)=2$ and $g(u)=0$. Hence, $g$ restricted to $V(T')$ is a TR2DF on $T'$ with weight $\omega(g)=\gamma_{t\{R2\}}(T)$, but it is not a $\gamma_{t\{R2\}}(T')$-function because $v\in S_1(T')$. So, $\gamma_{t\{R2\}}(T')\leq \gamma_{t\{R2\}}(T)-1$.

Moreover, let $f$ be a $\gamma_{tR}(T')$-function. Since $\gamma_{t\{R2\}}(T')=\gamma_{tR}(T')$ and $v\in S_1(T')$, we obtain that $f(v)=1$. Hence, $f$ can be extended to a TRDF on $T$ by assigning the weight $1$ to $u$. Thus, $\gamma_{tR}(T)\leq\gamma_{tR}(T')+1$. By using Proposition \ref{prop-inequalities}, inequalities above and inductive hypothesis, we obtain that
$\gamma_{t\{R2\}}(T)\leq \gamma_{tR}(T)\leq\gamma_{tR}(T')+1=\gamma_{t\{R2\}}(T')+1\leq \gamma_{t\{R2\}}(T)$.

Therefore, we must have equality throughout the inequality chain above. In particular, $\gamma_{t\{R2\}}(T)=\gamma_{tR}(T)$.\\

\noindent
\textbf{Case 4.} $T$ is obtained from $T'$ by Operation $F_4$. Assume $T$ is obtained from $T'$ by adding a path $uu_1$ and the edge $uv$, where $v\in S_{adj}(T')$. Notice that every TRDF on $T'$ can be extended to a TRDF on $T$ by assigning the weight $1$ to $u$ and $u_1$. Hence, by Proposition \ref{prop-inequalities}, the statement above and the inductive hypothesis, we obtain
\begin{equation}\label{eq11}
\gamma_{t\{R2\}}(T)\leq \gamma_{tR}(T)\leq \gamma_{tR}(T')+2=\gamma_{t\{R2\}}(T')+2.
\end{equation}

We now show that $\gamma_{t\{R2\}}(T')\leq \gamma_{t\{R2\}}(T)-2$. Let $w\in N(v)\cap S(T')$. By Observation \ref{obs-sup-adj}, there exists a $\gamma_{t\{R2\}}(T)$-function $f$ such that $f(v)=f(w)=f(u)=2$ and $f(u_1)=0$.  Hence, $f$ restricted to $V(T')$ is a TR2DF on $T'$, which implies that $\gamma_{t\{R2\}}(T')\leq f(V(T'))=\gamma_{t\{R2\}}(T)-2$, as desired. In consequence, we must have equality throughout the inequality chain (\ref{eq11}). In particular, $\gamma_{t\{R2\}}(T)=\gamma_{tR}(T)$.\\

\noindent
\textbf{Case 5.} $T$ is obtained from $T'$ by Operation $F_5$.
Assume $T$ is obtained from $T'$ by identifying the vertex $u$ of path $u_1uu_2$ and the vertex $v$, where $v\in L_{tR,1}(T')$. Notice that there exists a $\gamma_{tR}(T')$-function $g$ satisfying that $g(v)=1$. So, $g$ can be extended to a TRDF on $T$ be assigning the weight $2$ to $u$ and the weight $0$ to $u_1$ and $u_2$. Therefore, by Proposition \ref{prop-inequalities}, the statement above and the hypothesis, we obtain $\gamma_{t\{R2\}}(T)\leq \gamma_{tR}(T)\leq \gamma_{tR}(T')+1=\gamma_{t\{R2\}}(T')+1$.

Moreover, by Observation \ref{obs-strong}, there exists a $\gamma_{t\{R2\}}(T)$-function $f$ satisfying that $f(v)=2$ and $f(u_1)=f(u_2)=0$. Notice that the function $f'$, defined by $f'(v)=1$ and $f'(x)=f(x)$ whenever $x\in V(T')\setminus \{v\}$, is a TR2DF on $T'$. So $\gamma_{t\{R2\}}(T')\leq \omega(f')=\gamma_{t\{R2\}}(T)-1$. As a consequence, we must have equality throughout the inequality chain above. In particular, $\gamma_{t\{R2\}}(T)=\gamma_{tR}(T)$.\\

\noindent
\textbf{Case 6.} $T$ is obtained from $T'$ by Operation $F_6$.
Assume $T$ is obtained from $T'$ by adding a path $uu_1$ and the edge $uv$, where $v$ is a near stable vertex belonging to $L(T')\cup SS(T')$. Let $s$ be a support vertex adjacent to $v$ in $T'$. Again, notice that every TRDF on $T'$ can be extended to a TRDF on $T$ by assigning the weight $1$ to $u$ and $u_1$. Hence, by the statement above, Proposition \ref{prop-inequalities} and the inductive hypothesis, we obtain the inequality chain (\ref{eq11}). We now show that $\gamma_{t\{R2\}}(T')\leq \gamma_{t\{R2\}}(T)-2$. For this, we consider the next two cases.\\

\noindent
\textit{Case 6.1.} $v\in L(T')$. Since $u\in S(T)$ and $\delta_T(v)=\delta_T(u)=2$, there exists a $\gamma_{t\{R2\}}(T)$-function $f$ satisfying that $f(u)=f(u_1)=1$, $f(v)=0$ and $f(s)>0$.  If $f$ restricted to $V(T')$ is a TR2DF on $T'$, then $\gamma_{t\{R2\}}(T')\leq f(V(T'))=\gamma_{t\{R2\}}(T)-2$. Conversely, suppose that $f$ restricted to $V(T')$ is not a TR2DF on $T'$. So, $f$ restricted to $V(T')$ is a near-TR2DF relative to $v$ on $T'$. Thus, as $v$ is a near stable vertex of $T'$, and so, $\gamma_{t\{R2\}}(T')= \gamma_{t\{R2\}}^n(T';v)\leq f(V(T'))=\gamma_{t\{R2\}}(T)-2$, as desired.\\

\noindent
\textit{Case 6.2.} $v\in SS(T')$. Let $f$ be a $\gamma_{t\{R2\}}(T)$-function such that $f(u_1)$ is minimum. Hence $f(u)+f(u_1)=2$ and $f(s)>0$ since $s$ is a support of $T$. If $f$ restricted to $V(T')$ is a TR2DF on $T'$, then $\gamma_{t\{R2\}}(T')\leq f(V(T'))=\gamma_{t\{R2\}}(T)-2$. Conversely, suppose that $f$ restricted to $V(T')$ is not a TR2DF on $T'$. Hence $f(v)=0$, implying that $f$ restricted to $V(T')$ is a near-TR2DF relative to $v$ on $T'$. Also, as $v$ is a near stable vertex of $T'$, it follows that $\gamma_{t\{R2\}}(T')=\gamma_{t\{R2\}}^n(T';v)\leq f(V(T'))=\gamma_{t\{R2\}}(T)-2$, as desired.

In consequence, we must have equality throughout the inequality chain (\ref{eq11}). In particular, $\gamma_{t\{R2\}}(T)=\gamma_{tR}(T)$.\\

\noindent
\textbf{Case 7.} $T$ is obtained from $T'$ by Operation $F_7$.
Assume $T$ is obtained from $T'$ by adding a path $uu_1u_2$ and the edge $uv$, where $v\in W_0(T')$. Let $g$ be a $\gamma_{tR}(T')$-function. Hence, $g$ can be extended to a TRDF on $T$ be assigning the weight $1$ to $u$, $u_1$ and $u_2$. Therefore, by Proposition \ref{prop-inequalities}, the statement above and the hypothesis, we obtain
\begin{equation}\label{eq-9}
\gamma_{t\{R2\}}(T)\leq \gamma_{tR}(T)\leq \gamma_{tR}(T')+3=\gamma_{t\{R2\}}(T')+3.
\end{equation}

We now show that $\gamma_{t\{R2\}}(T')\leq \gamma_{t\{R2\}}(T)-3$.  Let $f(V_0,V_1,V_2)$ be a $\gamma_{t\{R2\}}(T)$-function such that $|V_2|$ is minimum. Hence $f(u_1)+f(u_2)=2$. If $f(u)=0$, then $f(v)>0$ and also, $f$ restricted to $V(T')$ is a TR2DF on $T'$. As $v\in W_0(T')$, $f$ is not a $\gamma_{t\{R2\}}(T')$-function. Hence, $\gamma_{t\{R2\}}(T')\leq f(V(T'))-1\leq \gamma_{t\{R2\}}(T)-3$, as desired.

Now, we suppose that $f(u)>0$. In this case, we observe that $f(u)=1$ since $|V_2|$ is minimum. If $f(v)=0$, then the function $f'$, defined by $f'(v)=f(u)$ and $f'(x)=f(x)$ whenever $x\in V(T')\setminus \{v\}$, is a TR2DF on $T'$ such that $\omega(f')\leq\gamma_{t\{R2\}}(T)-2$. On the other hand, if $f(v)>0$, then
observe that $f(N(v)\setminus\{u\})=0$. Otherwise, if there exists $z\in N(v)\setminus \{u\}$ such that $f(z)>0$, then $f(u)=0$, which is a contradiction. Now, notice that the function $f''$, defined by $f''(w)=f(u)=1$ for some $w\in N(v)\setminus\{u\}$ and $f''(x)=f(x)$ whenever $x\in V(T')\setminus \{w\}$, is a TR2DF on $T'$ such that $\omega(f')\leq\gamma_{t\{R2\}}(T)-2$. Again, as $v\in W_0(T')$, $f'$ and $f''$ are not $\gamma_{t\{R2\}}(T')$-functions. Hence, $\gamma_{t\{R2\}}(T')\leq \gamma_{t\{R2\}}(T)-3$, as desired.

In consequence, we must have equality throughout the inequality chain (\ref{eq-9}). In particular, $\gamma_{t\{R2\}}(T)=\gamma_{tR}(T)$.
\end{proof}

We now turn our attention to the opposite direction concerning the theorem above. That is, we show that if a tree $T$ satisfies  $\gamma_{t\{R2\}}(T)=\gamma_{tR}(T)$, then it belongs to the family $\mathcal{F}$.

\begin{theorem}\label{lem-left-1}
Let $T$ be a tree. If $\gamma_{t\{R2\}}(T)=\gamma_{tR}(T)$, then $T \in \mathcal{F}$.
\end{theorem}

\begin{proof}
First, we say that a tree $T$ belongs to the family $\mathcal{T}$ if $\gamma_{t\{R2\}}(T)=\gamma_{tR}(T)$. We proceed by induction on the order $n\geq 2$ of the trees $T\in \mathcal{T}$. If $T$ is a star, then $\gamma_{t\{R2\}}(T)=\gamma_{tR}(T)$. Thus, $T$ can be obtained from $P_2$ by first applying Operation $F_3$, thereby producing a path $P_3$ and then doing repeated applications of Operation $F_2$. Therefore, $T\in \mathcal{F}$. This establishes the base case. We assume now that $k\geq 3$ is an integer and that each tree $T'\in \mathcal{T}$ with $|V(T')| < k$ satisfies that $T'\in \mathcal{F}$. Let $T\in \mathcal{T}$ be a tree with $|V(T)|=k$  and we may assume that $diam(T)\geq 3$.

First, suppose that $diam(T)=3$. Therefore, $T$ is a double star $S_{x,y}$ for some integers $x \geq y \geq 1$. If $T\cong P_4$, then $T$ can be obtained from a path $P_2$ by applying Operation $F_4$. If $T\cong S_{x,y}$ with $x\geq y\geq 1$ ($T\not\cong P_4$), then $T$ can be obtained from a path $P_2$ by first applying Operation $F_4$, thereby producing a path $P_4$ and then doing repeated applications of Operation $F_2$ in both support vertices of $P_4$. Therefore, $T\in \mathcal{F}$.

We may now assume that $diam(T)\geq 4$, and we root the tree $T$ at a vertex $r$ located at the end of a longest path in $T$. Let $h$ be a vertex at maximum distance from $r$. Notice that, necessarily, $r$ and $h$ are leaves (and diametral vertices). Let $s$ be the parent of $h$; let $v$ be the parent of $s$; let $w$ be the parent of $v$; and let $z$ be the parent of $w$. Notice that all these vertices exist since $diam(T)\geq 4$, and it could happen $z=r$. Since $h$ is a vertex at maximum distance from the root $r$, every child of $s$ is a leaf. We proceed further with the following claims.\\

\noindent
{\bf Claim I.} If $\delta_T(s)\geq 4$, then $T\in \mathcal{F}$.\\

\noindent
{\bf Proof.} Suppose that $\delta_T(s)\geq 4$ and let $T'=T-h$. Hence, $\delta_{T'}(s)\geq 3$ and consequently, $s\in S_s(T')$, since every child of $s$ is a leaf vertex. Therefore, by Observation \ref{obs-strong-var-1}, there exists a $\gamma_{t\{R2\}}(T')$-function,  that assigns the weight $2$ to $s$. The function above can be extended to a TR2DF on $T$ by assigning the weight $0$ to  $h$, implying that $\gamma_{t\{R2\}}(T)\leq \gamma_{t\{R2\}}(T')$. Thus, by Proposition \ref{prop-inequalities}, Observation \ref{obs-lem-tR2-tR}, hypothesis and inequality above, we obtain $\gamma_{t\{R2\}}(T')\leq \gamma_{tR}(T')\leq \gamma_{tR}(T)=\gamma_{t\{R2\}}(T)\leq \gamma_{t\{R2\}}(T')$. Thus, we must have equality throughout this inequality chain. In particular $\gamma_{t\{R2\}}(T')=\gamma_{tR}(T')$. Applying the inductive hypothesis to $T'$, it follows that $T'\in \mathcal{F}$. Since $s\in S_s(T')$, by using Observation \ref{obs-strong-var-1}, we deduce that $s\in S_{tR,2}(T')$. Therefore, $T$ can be obtained from $T'$ by applying Operation $F_2$, and consequently, $T\in \mathcal{F}$. $(\square)$ \\

By Claim I, we may henceforth assume that $|N(x)\cap L(T)|=2$ for every strong support vertex $x$ of $T$.\\

\noindent
{\bf Claim II.} If $\delta_T(s)=3$ and $\delta_T(v)\geq 3$, then $T\in \mathcal{F}$.\\

\noindent
{\bf Proof.} Suppose that $\delta_T(s)=3$ and $\delta_T(v)\geq 3$. Thus, $s$ is a strong support vertex and has two leaf neighbors, say $h,h_1$. Moreover, observe that $v$ has at least one child, say $s'$, different from $s$, and also, $s'$ is either a leaf vertex or a support vertex of $T$. By using Theorem \ref{teo-tR2-tR}, there exists a $\gamma_{t\{R2\}}(T)$-function $f(V_0,V_1,V_2)$ with $V_{0,1}=\emptyset$, and without loss of generality, we assume that $|V_2|$ is maximum. Notice that $f$ is a $\gamma_{tR}(T)$-function as well. Now, we differentiate the following cases.\\

\noindent
\textit{Case 1.} $s'\in L(T)$. In such situation, $v\in S(T)$, and so, $f(v)=f(s)=2$ and $f(s')=f(h)=f(h_1)=0$. Let $T'=T-h$. Since $v,s\in S_{adj}(T')$, by Observation \ref{obs-sup-adj}, there exists a $\gamma_{t\{R2\}}(T')$-function $g$ satisfying that $g(v)=g(s)=2$. So, $g$ can be extended to a TR2DF on $T$ by assigning the weight $0$ to $h$. Hence $\gamma_{t\{R2\}}(T)\leq \gamma_{t\{R2\}}(T')$.
Consequently, by inequality above, Proposition \ref{prop-inequalities}, Observation \ref{obs-lem-tR2-tR} and hypothesis, we obtain $\gamma_{t\{R2\}}(T)\leq \gamma_{t\{R2\}}(T')\leq \gamma_{tR}(T')\leq \gamma_{tR}(T)=\gamma_{t\{R2\}}(T)$. Therefore, we must have equality throughout this inequality chain. In particular, $\gamma_{t\{R2\}}(T')=\gamma_{tR}(T')$. Applying the inductive hypothesis to $T'$, it follows that $T'\in \mathcal{F}$.

As another consequence of the equality chain above, we obtain that $\gamma_{tR}(T')=\gamma_{tR}(T)$. This implies that $f$ restricted to $V(T')$ is a $\gamma_{tR}(T')$-function, which means $s\in S_{tR,2}(T')$. Therefore, $T$ can be obtained from $T'$ by applying  Operation $F_2$, and consequently, $T\in \mathcal{F}$. $(\square)$ \\

\noindent
\textit{Case 2.} $s'\in S_s(T)$. Observe that $f(s')=f(s)=2$, $f(h)=0$ and $f(v)>0$. Let $T'=T-h$. Since $s'\in S_s(T')$, by Observation \ref{obs-strong-var-1}, there exists a $\gamma_{t\{R2\}}(T')$-function $g$ satisfying that $g(s')=2$. So, without loss of generality, we can assume that $g(v)>0$, implying that $g(s)=2$ and $g(h_1)=0$. Thus, $g$ can be extended to a TR2DF on $T$ by assigning the weight $0$ to $h$. Hence $\gamma_{t\{R2\}}(T)\leq \gamma_{t\{R2\}}(T')$. Consequently, by the inequality above, Proposition \ref{prop-inequalities}, Observation \ref{obs-lem-tR2-tR} and the hypothesis, we obtain $\gamma_{t\{R2\}}(T)\leq \gamma_{t\{R2\}}(T')\leq \gamma_{tR}(T')\leq \gamma_{tR}(T)=\gamma_{t\{R2\}}(T)$. Therefore, we must have equality throughout this inequality chain. In particular, $\gamma_{t\{R2\}}(T')=\gamma_{tR}(T')$. Applying the inductive hypothesis to $T'$, it follows that $T'\in \mathcal{F}$.

As another consequence of equality chain above, we obtain that $\gamma_{tR}(T')=\gamma_{tR}(T)$. This implies that $f$ restricted to $V(T')$ is a $\gamma_{tR}(T')$-function, and so, $s\in S_{tR,2}(T')$. Therefore, $T$ can be obtained from $T'$ by applying Operation $F_2$, which leads to $T\in \mathcal{F}$. $(\square)$\\

\noindent
\textit{Case 3.} $s'\in S(T)\setminus S_s(T)$. Notice that $T_{s'}\cong P_2$ and let $T'=T-T_{s'}$. Since $v\in SS(T')\cap N(S_s(T'))$, $f$ restricted to $V(T')$ is a TRDF on $T'$ and also $f(V(T_{s'}))=2$. Hence $\gamma_{tR}(T')\leq \gamma_{tR}(T)-2$. Moreover, any TR2DF on $T'$ we can extended to a TR2DF on $T$ by assigning the weight $1$ to $s'$ and its leaf-neighbor. Thus $\gamma_{t\{R2\}}(T)\leq \gamma_{t\{R2\}}(T')+2$. So, by these previous inequalities, Proposition \ref{prop-inequalities} and the hypothesis, we deduce $\gamma_{t\{R2\}}(T)\leq \gamma_{t\{R2\}}(T')+2\leq \gamma_{tR}(T')+2\leq \gamma_{tR}(T)=\gamma_{t\{R2\}}(T)$. Therefore, we must have equality throughout this inequality chain. In particular, $\gamma_{t\{R2\}}(T')=\gamma_{tR}(T')$. Applying the inductive hypothesis to $T'$, it follows that $T'\in \mathcal{F}$.

Moreover, as $v$ is adjacent to the support vertex $s$, every near-TR2DF relative to $v$ on $T'$ can be extended to a TR2DF on $T$ by assigning the weight $1$ to $s'$ and to its leaf-neighbor. So, $\gamma_{t\{R2\}}(T)\leq \gamma_{t\{R2\}}^n(T';v)+2$. In addition, if $v$ is not a near stable vertex of $T'$, then $\gamma_{t\{R2\}}^n(T';v)< \gamma_{t\{R2\}}(T')$, implying that $\gamma_{t\{R2\}}(T)\leq
\gamma_{t\{R2\}}^n(T';v)+2<
\gamma_{t\{R2\}}(T')+2$, which is a contradiction with the related equality noticed above. Therefore, the semi-support vertex $v$ is a near stable vertex of $T'$, and therefore, $T$ can be obtained from $T'$ by applying Operation $F_6$, which means $T\in \mathcal{F}$. $(\square)$\\

\noindent
{\bf Claim III.} If $\delta_T(s)=3$ and $\delta_T(v)=2$, then $T\in \mathcal{F}$.\\

\noindent
{\bf Proof.}
Suppose that $\delta_T(s)=3$ and $\delta_T(v)=2$. Thus, $s$ is a strong support vertex and has two leaf neighbors, say $h,h_1$. By Observation \ref{obs-strong}, there exists a $\gamma_{tR}(T)$-function $f$ such that $f(s)=2$ and $f(h)=f(h_1)=0$, which implies that $f(v)>0$. Let $T'=T-\{h,h_1\}$. Notice that the function $f'$, defined by $f'(s)=1$ and $f'(x)=f(x)$ whenever $x\in V(T')\setminus\{s\}$, is a TRDF on $T'$. Hence $\gamma_{tR}(T')\leq \omega(f')=\gamma_{tR}(T)-1$. Moreover, as $v\in S(T')$  and $\delta_{T'}(v)=2$, there exists a $\gamma_{t\{R2\}}(T')$-function $g$ satisfying that $g(s)=g(v)=1$. So, $g$ can be extended to a TR2DF on $T$ by assigning the weight $2$ to $s$ and the weight $0$ to $h$ and $h_1$, implying that $\gamma_{t\{R2\}}(T)\leq \gamma_{t\{R2\}}(T')+1$. Thus, by Proposition \ref{prop-inequalities}, the hypothesis and the inequalities above, we obtain that $\gamma_{t\{R2\}}(T')\leq \gamma_{tR}(T')\leq \gamma_{tR}(T)-1=\gamma_{t\{R2\}}(T)-1\leq \gamma_{t\{R2\}}(T')$. Therefore, we must have equality throughout this inequality chain. In particular, $\gamma_{t\{R2\}}(T')=\gamma_{tR}(T')$. Applying the inductive hypothesis to $T'$, it follows that $T'\in \mathcal{F}$.

Moreover, by the equality noted before, we deduce that $f'$ is a $\gamma_{tR}(T')$-function. Thus $s\in L_{tR,1}(T')$. Therefore, $T$ can be obtained from $T'$ by applying Operation $F_5$, and consequently, $T\in \mathcal{F}$. $(\square)$\\

\noindent
{\bf Claim IV.} If $\delta_T(s)=2$ and $\delta_T(v)=2$, then $T\in \mathcal{F}$.\\

\noindent
{\bf Proof.} Suppose that $\delta_T(s)=2$ and $\delta_T(v)=2$. By Theorem \ref{teo-tR2-tR}, there exists a $\gamma_{t\{R2\}}(T)$-function $f$ with $V_{0,1}=\emptyset$ and, without loss of generality, we assume that $|V_2|$ is minimum. Notice that $f$ is a $\gamma_{tR}(T)$-function as well, and also $f(s)+f(h)=2$.

First, we suppose that $f(w)>0$. Let $T'=T-T_s=T-\{s,h\}$.
Notice that $f$ restricted to $V(T')$ is a TRDF on $T'$. Hence $\gamma_{tR}(T')\leq f(V(T'))=\gamma_{tR}(T)-2$. Moreover, any TR2DF on $T'$ can be extended to a TR2DF on $T$ by assigning the weight $1$ to $s$ and $h$, implying that $\gamma_{t\{R2\}}(T)\leq \gamma_{t\{R2\}}(T')+2$. So, by the inequalities above, Proposition \ref{prop-inequalities} and the hypothesis, we obtain that $\gamma_{t\{R2\}}(T)\leq \gamma_{t\{R2\}}(T')+2\leq \gamma_{tR}(T')+2\leq \gamma_{tR}(T)=\gamma_{t\{R2\}}(T)$. Therefore, we must have equality throughout this inequality chain. In particular, $\gamma_{t\{R2\}}(T')=\gamma_{tR}(T')$. Applying the inductive hypothesis to $T'$, it follows that $T'\in \mathcal{F}$.

Moreover, a minimum weight near-TR2DF relative to $v$ on $T'$ can be extended to a TR2DF on $T$ by assigning to $s$ and $h$ the weight $1$. Hence $\gamma_{t\{R2\}}(T)\leq \gamma_{t\{R2\}}^n(T';v)+2$. If $v$ is not a near stable vertex of $T'$, then  $\gamma_{t\{R2\}}^n(T';v)<\gamma_{t\{R2\}}(T')$, implying that $\gamma_{t\{R2\}}(T)\leq \gamma_{t\{R2\}}^n(T';v)+2<\gamma_{t\{R2\}}(T')+2$, which is a contradiction with the related equality noted before. Therefore, $v$ is both a near stable vertex  and a leaf of $T'$. Thus, $T$ can be obtained from the tree $T'$ by applying Operation $F_6$, and consequently, $T\in \mathcal{F}$.

From now on, we suppose that $f(w)=0$. Hence $f(v)=f(s)=f(h)=1$. We consider the tree $T''=T-T_v=T-\{v,s,h\}$. Notice that any TR2DF on $T''$ can be extended to a TR2DF on $T$ by assigning the weight $1$ to $v$, $s$ and $h$, implying that $\gamma_{t\{R2\}}(T)\leq \gamma_{t\{R2\}}(T'')+3$.

On the other hand, notice that $f$ restricted to $V(T'')$ is a TRDF on $T''$. Hence $\gamma_{tR}(T'')\leq f(V(T''))=\gamma_{tR}(T)-3$. Consequently, by the previous inequalities, Proposition \ref{prop-inequalities} and the hypothesis, we obtain that $\gamma_{t\{R2\}}(T)\leq \gamma_{t\{R2\}}(T'')+3\leq \gamma_{tR}(T'')+3\leq \gamma_{tR}(T)=\gamma_{t\{R2\}}(T)$. Therefore, we must have equality throughout this inequality chain. In particular, $\gamma_{t\{R2\}}(T'')=\gamma_{tR}(T'')$. Also, note that $\gamma_{t\{R2\}}(T)=\gamma_{t\{R2\}}(T'')+3$. Applying the inductive hypothesis to $T''$, it follows that $T''\in \mathcal{F}$.

Moreover, we suppose there exists a $\gamma_{t\{R2\}}(T'')$-function $g$ satisfying that $g(w)>0$. Observe that $g$ can be extended to a TR2DF on $T$ by assigning the weight $0$ to $v$ and the weight $1$ to $s$ and $h$. Hence $\gamma_{t\{R2\}}(T)\leq \omega(g)+2=\gamma_{t\{R2\}}(T'')+2$, which is a contradiction with the related equality noticed above. Therefore $w\in W_0(T'')$ and so, $T$ can be obtained from the tree $T''$ by applying Operation $F_7$. Consequently, $T\in \mathcal{F}$. $(\square)$\\

\noindent
{\bf Claim V.} If $\delta_T(s)=2$ and $\delta_T(v)\geq 3$, then $T\in \mathcal{F}$.\\

\noindent
{\bf Proof.} Suppose that $\delta_T(s)=2$ and $\delta_T(v)\geq 3$. Clearly, $v$ has at least one child, say $s'$, different from $s$, implying that $s'$ is either a support vertex or a leaf vertex of $T$. Now, we differentiate the following cases.\\

\noindent
\textit{Case 1.} $s'\in S(T)$. By Theorem \ref{teo-tR2-tR}, there exists a $\gamma_{t\{R2\}}(T)$-function $f$ with $V_{0,1}=\emptyset$ and, without loss of generality, we assume that $|V_2|$ is maximum. Notice that $f$ is a $\gamma_{tR}(T)$-function as well, and also, $f(s)+f(h)=2$. As $s'\in S(T)$, $f$ restricted to $T'=T-\{s,h\}$ is a TRDF on $T'$. Hence $\gamma_{tR}(T') \leq \gamma_{tR}(T)-2$. Moreover, any TR2DF on $T'$ can be extended to a TR2DF on $T$ by assigning the weight $1$ to $s$ and $h$. Hence $\gamma_{t\{R2\}}(T) \leq \gamma_{t\{R2\}}(T')+2$. Therefore, by the inequalities above, Proposition \ref{prop-inequalities} and the hypothesis, we obtain $\gamma_{t\{R2\}}(T) \leq \gamma_{t\{R2\}}(T')+2 \leq \gamma_{tR}(T')+2\leq \gamma_{tR}(T) = \gamma_{t\{R2\}}(T)$. Thus, we must have equality throughout this inequality chain. In particular, $\gamma_{t\{R2\}}(T') = \gamma_{tR}(T')$. Applying the inductive hypothesis to $T'$, it follows that $T' \in \mathcal{F}$.

If $v\in S(T')$, then $v\in S_{adj}(T')$. So $T$ can be obtained from $T'$ by applying Operation $F_4$, and consequently, $T\in \mathcal{F}$.

If $v\notin S(T')$, then $v\in SS(T')$. Now, we prove that $v$ is a near stable vertex of $T'$. Notice that a minimum weight near-TR2DF relative to $v$ on $T'$ can be extended to a TR2DF on $T$ by assigning to $s$ and $h$ the weight $1$. So $\gamma_{t\{R2\}}(T)\leq \gamma_{t\{R2\}}^n(T';v)+2$. If $v$ is not a near stable vertex of $T'$, then  $\gamma_{t\{R2\}}^n(T';v)<\gamma_{t\{R2\}}(T')$, implying that $\gamma_{t\{R2\}}(T)\leq \gamma_{t\{R2\}}^n(T';v)+2<\gamma_{t\{R2\}}(T')+2$, which is a contradiction with the related equality noted before. Therefore, $v$ is a near stable vertex of $T'$, as desired. Thus, $T$ can be obtained from  $T'$ by applying Operation $F_6$, and consequently, $T\in \mathcal{F}$. $(\square)$\\

By the case above, we may henceforth assume that every child of $v$ is a leaf of $T$.\\

\noindent
\textit{Case 2.} $s'\in L(T)$ and $v\in S_s(T)$. We consider the tree $T'=T-s'$. Notice that $v\in S_{adj}(T')$. Hence, by Observation \ref{obs-sup-adj}, there exists a  $\gamma_{t\{R2\}}(T')$-function $f$ such that $f(v)=f(s)=2$. So, $f$ can be extended to a TR2DF on $T$ by assigning the weight $0$ to $s'$. Thus $\gamma_{t\{R2\}}(T)\leq \gamma_{t\{R2\}}(T')$ and by using Proposition \ref{prop-inequalities}, Observation \ref{obs-lem-tR2-tR} and the hypothesis, we obtain $\gamma_{t\{R2\}}(T) \leq \gamma_{t\{R2\}}(T') \leq \gamma_{tR}(T') \leq \gamma_{tR}(T) = \gamma_{t\{R2\}}(T)$. Therefore, we must have equality throughout this inequality chain. In particular, $\gamma_{t\{R2\}}(T') = \gamma_{tR}(T')$. Applying the inductive hypothesis to $T'$, it follows that $T' \in \mathcal{F}$. As another consequence of equality chain above, we obtain $\gamma_{tR}(T') = \gamma_{tR}(T)$. By Observation \ref{obs-strong}, there exists a $\gamma_{tR}(T)$-function $g$ such that $g(v)=2$ and $g(s')=0$. Since $\gamma_{tR}(T') = \gamma_{tR}(T)$, $g$ restricted to $V(T')$ is a $\gamma_{tR}(T')$-function. Hence $v\in S_{tR,2}(T')$. Therefore, $T$ can be obtained from $T'$ by applying Operation $F_2$, and consequently, $T\in \mathcal{F}$. $(\square)$\\

\noindent
\textit{Case 3.} $s'\in L(T)$ and $v\in S(T)\setminus S_s(T)$. First, we suppose that $w\in S(T)$. Let $T'=T-T_s=T-\{s,h\}$. By using a similar procedure as in Case 1 of Claim V ($v\in S(T')$),  we obtain $\gamma_{t\{R2\}}(T')=\gamma_{tR}(T')$ and $v\in S_{adj}(T')$.  Hence, by applying the inductive hypothesis to $T'$, it follows that $T' \in \mathcal{F}$. Therefore, $T$ can be obtained from $T'$ by Operation $F_4$, and consequently, $T\in \mathcal{F}$.

From now on, we assume that $w\notin S(T)$. Let $f(V_0,V_1,V_2)$ be a $\gamma_{tR}(T)$-function such that $f(w)$ is minimum among all $\gamma_{tR}(T)$-functions which satisfy that $|S(T)\cap V_2|$ and $|L_s(T)\cap V_0|$ are maximum. Hence $f(v)=f(s)=2$. Next, we analyse the two possible scenarios.\\

\noindent
\textit{Subcase 3.1.} $f(w)>0$. Since $N(w)\setminus \{z\}\subset S(T)\cup SS(T)$ and $f(v)=2$, it is easy to check that $f(w)=1$ and $N(w)\cap S_s(T)\neq \emptyset$. Let $v'\in N(w)\cap S_s(T)$, $N(v')\cap L(T)=\{h_1,h_2\}$ and $T''=T-\{h_1,h_2\}$.

Notice that the function $f'$, defined by $f'(v')=1$ and $f'(x)=f(x)$ whenever $x\in V(T'')\setminus\{v'\}$, is a TRDF on $T''$. Hence $\gamma_{tR}(T'')\leq \omega(f')=\gamma_{tR}(T)-1$.  Moreover, as $v\in S(T'')$ and $N(w)\setminus \{z,v'\}\subset S(T'')\cup SS(T'')$, there exists a $\gamma_{t\{R2\}}(T'')$-function $g$ such that $g(w)=g(v')=1$. So, $g$ can be extended to a TR2DF on $T$ by re-assigning the weight $2$ to $v'$ and by assigning the weight $0$ to $h_1$ and $h_2$, which implies that $\gamma_{t\{R2\}}(T)\leq \gamma_{t\{R2\}}(T'')+1$.

Thus, by Proposition \ref{prop-inequalities}, the inequalities above and the hypothesis, we obtain $\gamma_{t\{R2\}}(T'')\leq \gamma_{tR}(T'')\leq \gamma_{tR}(T)-1=\gamma_{t\{R2\}}(T)-1\leq \gamma_{t\{R2\}}(T'')$. Therefore, we must have equality throughout this inequality chain. In particular, $\gamma_{t\{R2\}}(T'')=\gamma_{tR}(T'')$. Applying the inductive hypothesis to $T''$, it follows that $T''\in \mathcal{F}$. Also, by the equality noted before, we deduce that $f'$ is a $\gamma_{tR}(T'')$-function. Thus $v'\in L_{tR,1}(T'')$. Therefore, $T$ can be obtained from $T''$ by applying Operation $F_5$, and consequently, $T\in \mathcal{F}$. $(\square)$\\

\noindent
\textit{Subcase 3.2.} $f(w)=0$. Notice that $N(w)\cap S_s(T)=\emptyset$. Now, we consider that $w$ has a child, say $v'$, different from $v$. First, we suppose that $S_s(T)\cap V(T_{v'})\neq \emptyset$. Let $x\in S_s(T)\cap V(T_{v'})$, $h_x\in N(x)\cap L(T)$ and $T''=T-h_x$. Again, by using a similar procedure as in Case 2 of Claim V, we obtain $\gamma_{t\{R2\}}(T'')=\gamma_{tR}(T'')$ and $x\in S_{tR,2}(T'')$.  Hence, by applying the inductive hypothesis to $T''$, it follows that $T'' \in \mathcal{F}$. Therefore, $T$ can be obtained from $T''$ by applying Operation $F_2$, and consequently, $T\in \mathcal{F}$.


Thus, we may assume that $S_s(T)\cap V(T_w)=\emptyset$. If $T_{v'}$ is isomorphic to $P_2$ or $P_3$, then, by using a similar reasoning as in Case 1 of Claim V ($v\in SS(T')$) or Claim IV ($f(w)=0$), respectively, we obtain that $T''=T-T_{v'}\in \mathcal{F}$. Therefore, $T$ can be obtained from $T''$ by applying Operation $F_6$ or Operation $F_7$, respectively. Consequently, $T\in \mathcal{F}$.

Hence, we may assume that  for every child $x$ of $w$, the tree $T_x$ is not isomorphic to $P_2$ or $P_3$. Thus, it is easy to check that $T_w\cong R_r$. Let $T'''=T-T_w$. Since $f(w)=0$ and $w\notin S(T)$, the function $f$ restricted to $V(T''')$ is a TRDF on $T'''$. So, $\gamma_{tR}(T''')\leq f(V(T'''))=\gamma_{tR}(T)-4r$. Moreover, any TR2DF on $T'''$ can be extended to a TR2DF on $T$ by assigning the weight $2$ to every support vertex and the weight $0$ to another vertices of $T_w$. Thus, $\gamma_{t\{R2\}}(T)\leq \gamma_{t\{R2\}}(T''')+4r$, and by using the inequalities above, Proposition \ref{prop-inequalities} and the hypothesis, we obtain $\gamma_{t\{R2\}}(T)\leq \gamma_{t\{R2\}}(T''')+4r \leq \gamma_{tR}(T''')+4r\leq \gamma_{tR}(T) = \gamma_{t\{R2\}}(T)$. Therefore, we must have equality throughout this inequality chain. In particular, $\gamma_{t\{R2\}}(T''') = \gamma_{tR}(T''')$. Applying the inductive hypothesis to $T'''$, it follows that $T'''\in \mathcal{F}$. Therefore, $T$ can be obtained from $T'''$ by applying Operation $F_1$, and consequently, $T\in \mathcal{F}$, which completes the proof.
\end{proof}

As an immediate consequence of Theorems \ref{lem-right-1} and \ref{lem-left-1}, we have the following characterization.

\begin{theorem}\label{teo-char-R2-R}
A tree $T$ of order $n\geq 2$ satisfies that $\gamma_{t\{R2\}}(T)=\gamma_{tR}(T)$ if and only if $T \in \mathcal{F}$.
\end{theorem}

To conclude this subsection, we next give a characterization of trees $T$ with $\gamma_{t\{R2\}}(T)=2\gamma_t(T)$. In \cite{henning2006}, a family $\mathcal{T}$ of trees $T$ with $\gamma_t(T)=\gamma(T)$ were characterized. Hence, as a consequence of the statement above and Theorems \ref{teo-tR2-2t} and \ref{teo-char-R2-R}, the next characterization follows.

\begin{theorem}\label{teo-char-tR2-2t}
A tree $T$ of order $n\geq 2$ satisfies that $\gamma_{t\{R2\}}(T)=2\gamma_t(T)$ if and only if $T \in \mathcal{F}\cap \mathcal{T}$.
\end{theorem}

We must remark that our characterization is strongly based on the computability of the sets $S_{tR,2}(T_i)$, $W_0(T_i)$, $S_1(T_i)$ and $L_{tR,1}(T_i)$, for a given tree $T_i$, in order to construct a new element $T_{i+1}$ of the family $\mathcal{F}$. It is probably hard to find such sets for the tree $T_i$ regardless which is the operation made to construct such $T_i$. In this sense, as a continuation of this work, it would be desirable a future discussion on how one of these sets can be obtained for a given tree $T_i$, and on whether a connection between such set in $T_i$ and the corresponding one in $T_{i+1}$ exists.

\section{Computational results}

In order to present our complexity results we need to introduce the following construction. Given a graph $G$ of order $n$ and $n$ copies of the star graph $K_{1,4}$, the graph $H_G$ is constructed by adding edges between the $i^{th}$-vertex of $G$ and one leaf vertex of the $i^{th}$-copy of $K_{1,4}$. See Figure \ref{fig-complex} for an example.

\begin{figure}[h]
\centering
\begin{tikzpicture}[scale=.7, transform shape]

\node [draw, shape=circle] (a1) at  (0,0) {};
\node [draw, shape=circle] (a2) at  (3,0) {};
\node [draw, shape=circle] (a3) at  (6,0) {};
\node [draw, shape=circle] (a4) at  (9,0) {};

\node [draw, shape=circle] (a5) at  (1,1) {};
\node [draw, shape=circle] (a6) at  (4,1) {};
\node [draw, shape=circle] (a7) at  (7,1) {};
\node [draw, shape=circle] (a8) at  (10,1) {};

\node [draw, shape=circle] (b13) at  (1.5,2.5) {};
\node [draw, shape=circle] (b14) at  (4.5,2.5) {};
\node [draw, shape=circle] (b15) at  (7.5,2.5) {};
\node [draw, shape=circle] (b16) at  (10.5,2.5) {};

\node [draw, shape=circle] (a17) at  (0.5,4) {};
\node [draw, shape=circle] (a18) at  (3.5,4) {};
\node [draw, shape=circle] (a19) at  (6.5,4) {};
\node [draw, shape=circle] (a20) at  (9.5,4) {};

\node [draw, shape=circle] (b17) at  (2.5,4) {};
\node [draw, shape=circle] (b18) at  (5.5,4) {};
\node [draw, shape=circle] (b19) at  (8.5,4) {};
\node [draw, shape=circle] (b20) at  (11.5,4) {};

\node [draw, shape=circle] (c17) at  (1.5,4) {};
\node [draw, shape=circle] (c18) at  (4.5,4) {};
\node [draw, shape=circle] (c19) at  (7.5,4) {};
\node [draw, shape=circle] (c20) at  (10.5,4) {};

\draw(a5)--(a1)--(a2)--(a3)--(a4)--(a8);
\draw(a2)--(a6);
\draw(a3)--(a7);
\draw(c17)--(b13)--(a5);
\draw(c18)--(b14)--(a6);
\draw(c19)--(b15)--(a7);
\draw(c20)--(b16)--(a8);
\draw(a17)--(b13)--(b17);
\draw(a18)--(b14)--(b18);
\draw(a19)--(b15)--(b19);
\draw(a20)--(b16)--(b20);

\draw (a1) .. controls (0.3,-0.5) and (5.7,-0.5) .. (a3);
\draw (a2) .. controls (3.3,-0.5) and (8.7,-0.5) .. (a4);
\end{tikzpicture}
\caption{The graph $H_G$ where $G$ is a complete graph minus one edge.}
\label{fig-complex}
\end{figure}
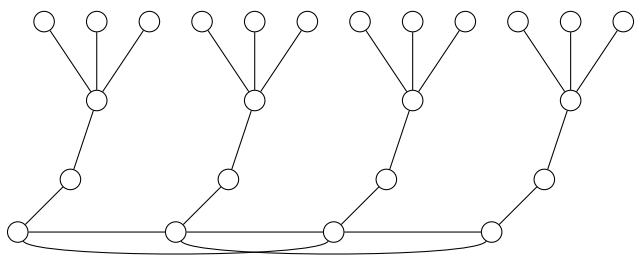

It is well-known that the domination Problem is NP-complete, even when restricted to bipartite graphs (see Dewdney \cite{NP5}) or chordal graphs (see Booth \cite{NP1} and Booth and Johnson \cite{NP2}). We use this result to prove the main result of this section, which is the complexity analysis of the following decision problem (total Roman  $\{2\}$-dominating function Problem (TR2DF-Problem for short)). To this end, we will demonstrate a polynomial time reduction of the domination Problem  to our TR2DF-Problem.\\

\begin{tabular}{l}
\hline
  \mbox{TR2DF-Problem }\\
  \mbox{Instance: A non trivial graph $H$ and a positive integer $j\leq |V(H)|$.}\\
  \mbox{Question: Does $H$ have a TR2DF of weight $j$ or less?}\\
\hline
\end{tabular}\\

\begin{observation}\label{obs-super-strong}
Let $G$ be a graph different from a star graph. If $v\in S_s(G)$  such that $|N(v)\cap L(G)|\geq 3$, then $f(v)=2$ for every $\gamma_{t\{R2\}}(G)$-function $f$.

\end{observation}

\begin{theorem}
TR2DF-Problem is NP-complete, even when restricted to bipartite or chordal graphs.
\end{theorem}

\begin{proof}
The problem is clearly in NP since verifying that a given function is indeed a TR2DF can be done in polynomial time.

We consider a graph $G$ without isolated vertices of order $n$ and construct the graph $H_G$. It is easy to see that this construction can be accomplished in polynomial time.  Also, notice that if the graph $G$ is a bipartite or chordal graph, then so too is $H_G$.

We next prove that $\gamma_{t\{R2\}}(H_G)=\gamma(G)+3|V(G)|$. For this, we first consider the function $f'(V_0',V_1',V_2')$ defined by $V_1'=A\cup SS(H_G)$ and $V_2'=S(H_G)$, where $A$ is a $\gamma(G)$-set. Notice that $f'$ is a TR2DF on $H_G$. So $\gamma_{t\{R2\}}(H_G)\leq |A \cup SS(H_G)|+2|S(H_G)|=\gamma(G)+3|V(G)|$.

On the other hand, let $v\in V(G)\subset V(H_G)$, and by $K_{1,4}^v$ we denote the copy of $K_{1,4}$ added to $v$. Let $s_v$ and $u_v$ be the support vertex and semi-support vertex of $H_G$ respectively, belonging to the copy $K_{1,4}^v$. Let $f(V_0,V_1,V_2)$ be a $\gamma_{t\{R2\}}(G)$-function satisfying that $|V_2|$ is minimum. Since $H_G$ is different from a star, and every support vertex is adjacent to three leaves, by Observation \ref{obs-super-strong}, we obtain that $f(S(H_G))=2|S(H_G)|$. Consequently, $f(s_v)=2$ and $f(V(K_{1,4}^v))\geq 3$. Hence, we can assume that $f(u_v)>0$. If $f(u_v)=2$, then the function $f''(V_0'',V_1'',V_2'')$, defined by $f''(u_v)=1$, $f''(v)=\min\{f(v)+1,2\}$ and $f''(x)=f(x)$ whenever $x\in V(H_G)\setminus \{u_v,v\}$, is a TR2DF on $H_G$ of weight $\gamma_{t\{R2\}}(H_G)$ and $|V_2''|<|V_2|$, which is a contradiction. Hence $f(u_v)=1$ for every $v\in V(G)$.

Notice that each vertex $v\in V(G)$ is adjacent to exactly one semi-support vertex of $H_G$. As $SS(H_G)\subseteq V_1$, it follows that $V(G)\subseteq V_0\cup V_1$ and also, $V_1\cap V(G)$ is a dominating set of $G$. Thus, $\gamma_{t\{R2\}}(H_G)=\omega(f)=|V_1\cap V(G)|+|SS(H_G)|+2|S(H_G)|\geq \gamma(G)+3|V(G)|$. As a consequence, it follows that $\gamma_{t\{R2\}}(H_G) = \gamma(G) + 3|V(G)|$, as required.

Now, for $j=k+3|V(G)|$, it is readily seen that $\gamma_{t\{R2\}}(H_G)\leq j$ if and only if $\gamma(G)\leq k$, which completes the proof.
\end{proof}

As a consequence of the result above we conclude that finding the total Roman $\{2\}$-domination number of graphs is NP-hard.

\end{document}